\documentclass{amsart}

\usepackage{amsmath, amssymb, amscd, mathrsfs}

\usepackage{xcolor}
\usepackage[linktocpage=true,bookmarksopen=true]{hyperref}
\usepackage[only,mapsfrom]{stmaryrd}

\newtheorem{theorem}{Theorem}[section]
\newtheorem{proposition}[theorem]{Proposition}
\newtheorem{definition}[theorem]{Definition}
\newtheorem{conjecture}[theorem]{Conjecture}
\theoremstyle{remark}
\newtheorem{remark}[theorem]{Remark}

\newtheorem*{acknowledgments}{Acknowledgments}

\numberwithin{equation}{section}

\newcommand{\Order}{\mathcal{O}}
\newcommand{\into}{\hookrightarrow}

\newcommand{\onto}{\twoheadrightarrow}
\newcommand{\isomto}{\overset{\sim}{\to}}
\newcommand{\tensor}{\mathbin{\otimes}}
\newcommand{\N}{\mathbb{N}}
\newcommand{\Z}{\mathbb{Z}}
\newcommand{\Q}{\mathbb{Q}}

\newcommand{\F}{\mathbb{F}}
\newcommand{\Gm}{\mathbf{G}_{m}}
\newcommand{\Ga}{\mathbf{G}_{a}}
\newcommand{\dirlim}{\varinjlim}
\newcommand{\invlim}{\varprojlim}
\newcommand{\sep}{\mathrm{sep}}
\newcommand{\Affine}{\mathbb{A}}
\newcommand{\Proj}{\mathbb{P}}
\newcommand{\incl}{\mathrm{incl}}
\newcommand{\RP}{\mathrm{RP}}
\newcommand{\RPS}{\mathrm{RPS}}
\newcommand{\sm}{\mathrm{sm}}
\newcommand{\et}{\mathrm{et}}
\DeclareMathOperator{\Hom}{Hom}
\DeclareMathOperator{\End}{End}

\DeclareMathOperator{\Ext}{Ext}
\DeclareMathOperator{\Spec}{Spec}
\DeclareMathOperator{\Ab}{Ab}
\DeclareMathOperator{\sheafhom}{\mathbf{Hom}}
\DeclareMathOperator{\sheafext}{\mathbf{Ext}}
\DeclareMathOperator{\Res}{Res}
\DeclareMathOperator{\dlog}{dlog}

\title{Existence of global N\'eron models beyond semi-abelian varieties}
\author{Otto Overkamp}
\address{Mathematisches Institut der Heinrich-Heine-Universit\"at D\"usseldorf, Universit\"atsstr. 1, 40225 D\"usseldorf, Germany}
\email{otto.overkamp@uni-duesseldorf.de}
\author{Takashi Suzuki}
\address{
	Department of Mathematics, Chuo University,
	1-13-27 Kasuga, Bunkyo-ku, Tokyo 112-8551, Japan
}
\email{tsuzuki@gug.math.chuo-u.ac.jp}
\date{June 9, 2026}

\begin{document}

\begin{abstract}
	We first prove Bosch-L\"utkebohmert-Raynaud's conjectures
	on existence of global N\'eron models of
	not necessarily semi-abelian algebraic groups
	in the perfect residue fields case.
	We then give a counterexample to the existence
	in the imperfect residue fields case.
	Finally,
	as a complement to the conjectures,
	we classify unirational wound unipotent groups
	``up to relative perfection'',
	again in the perfect residue fields case.
	The key ingredient for all these is
	the duality for relatively perfect unipotent groups.
\end{abstract}

\maketitle

\tableofcontents

%%%%%%%%%%%%%%%%%%%%%%%%%%%%%%%%%%%%%%%%%%%%%%%%%%%%%%%%%%%%%%%%%%%%%%%%%%%%%

\section{Introduction}

%%%%%%%%%%%%%%%%%%%%%%%%%%%%%%%%%%%%%%%%%%%%%%%%%%%%%%%%%%%%%%%%%%%%%%%%%%%%%

\subsection{Background}

Let $S$ be an irreducible Dedekind scheme
with function field $K$.
For a smooth group scheme $G_{K}$ over $K$,
a \emph{N\'eron model} of $G_{K}$ over $S$ is
a smooth group scheme $G$ over $S$
together with an isomorphism
$G \times_{S} \Spec K \cong G_{K}$
such that for any smooth $S$-scheme $X$,
any morphism $X \times_{S} \Spec K \to G_{K}$ over $K$
uniquely extends to a morphism $X \to G$ over $S$.
It is unique if it exists.
The basic reference is
Bosch-L\"utkebohmert-Raynaud's book \cite{BLR90}.%
\footnote{
	N\'eron models as defined above are called
	N\'eron lft-models
	(where lft stands for locally finite type)
	in \cite[Section 10]{BLR90}.
	N\'eron models as defined in \cite{BLR90} are
	further assumed to be quasi-compact.
}
In what follows,
we assume that $S$ is excellent
and $G_{K}$ is connected.
All group schemes below are assumed commutative.

The existence of a N\'eron model is known to be true
in some cases and false in some other cases
and is unknown in general.
A satisfactory answer to this problem is given when $S$ is local
(\cite[Section 10.2, Theorems 1 and 2]{BLR90}):
$G_{K}$ admits a N\'eron model
(resp.\ a quasi-compact N\'eron model)
if and only if it does not contain $\Ga$
(resp.\ $\Ga$ or the multiplicative group $\Gm$
after strict henselization).
For a general (global) $S$,
the N\'eron model of an abelian variety exists and is quasi-compact
by \cite[Section 1.4, Theorem 3]{BLR90}.
More generally, the N\'eron model of a semi-abelian variety
(an extension of an abelian variety by a torus) exists
by \cite[Section 10.1, Proposition 6 and Section 7.5, Proposition 1]{BLR90}
but may not be quasi-compact.
On the other hand, the N\'eron model of $G_{K}$ does not exist
if it contains $\Ga$ (\cite[Section 10.1, Proposition 8]{BLR90}).
These settle the problem when the characteristic of $K$ is zero.
For the positive characteristic case and for general $S$,
a mystery remained and \cite{BLR90} left it as a conjecture:

\begin{conjecture}[{\cite[Section 10.3, Conjectures I and II]{BLR90}}] \label{0012}
	Let $G_{K}$ be a connected smooth algebraic group over $K$.
	\begin{enumerate}
		\item \label{0047}
                If $G_{K}$ does not contain $\Ga$,
                then it admits a N\'eron model over $S$.
		\item \label{0048}
			If $G_{K}$ does not contain
                a non-zero unirational subgroup,%
			\footnote{
				A variety over $K$ is said to be unirational
				if there is a dominant rational map
				from an affine space $\Affine_{K}^{n}$ over $K$.
			}
			then it admits a quasi-compact N\'eron model over $S$.
	\end{enumerate}
\end{conjecture}

Part of the difficulties lies in the facts that
non-zero unipotent groups over imperfect fields
do not have to contain $\Ga$
(such unipotent groups are said to be \emph{wound})
and there is no good notion
of good reduction for unipotent groups
in positive characteristic.
These imply that there may be serious differences
between generic fibers and general fibers
when $S$ has infinitely many points.
For instance, there is a unipotent wound
unirational (hence connected) group over $K$
with a N\'eron model
\emph{all} of whose special fibers are disconnected
(Oesterl\'e's example
\cite[Section 10.1, Example 11]{BLR90}).
It is given by $(\Res_{K^{1 / p^{n}} / K} \Gm) / \Gm$
with $n \ge 1$ and $[K : K^{p}] < \infty$,
where $\Res$ denotes the Weil restriction functor.
Its N\'eron model is not quasi-compact.

Some partial results are known.
Statement \eqref{0048} is true
if $G_{K}$ admits a regular compactification
by \cite[Section 10.3, Theorem 5 (a)]{BLR90}
(and hence it is true under the assumption of
existence of resolutions of singularities).
The assumption of resolutions of singularities
should be a red herring
given that it is not needed to construct
(quasi-compact) N\'eron models of abelian varieties.
Both statements \eqref{0047} and \eqref{0048} are true
if $G_{K}$ is the Jacobian
of a proper geometrically reduced curve
by the first author (\cite[Theorems 1.2 and 1.3]{Ove24}).

%%%%%%%%%%%%%%%%%%%%%%%%%%%%%%%%%%%%%%%%%%%%%%%%%%%%%%%%%%%%%%%%%%%%%%%%%%%%%

\subsection{Main results}

In this paper, we prove the following:

\begin{theorem} \label{0050} \mbox{}
	\begin{enumerate}
		\item \label{0052}
			Conjecture \ref{0012} (both statements) is true
			if the residue fields of $S$
			at closed points are perfect.
		\item \label{0053}
			Conjecture \ref{0012} \eqref{0047} is
			false in general.
			More precisely,
			if $K$ has positive characteristic $p > 0$
			with $[K : K^{p}] < \infty$
			and $S$ has infinitely many points and
			an imperfect residue field at some closed point,
			then there exists a wound unipotent
			unirational group over $K$
			that does not admit a N\'eron model over $S$.
	\end{enumerate}
\end{theorem}

See Theorems \ref{0009} and \ref{0013}
for Statement \eqref{0052}
and Theorem \ref{0032} and Proposition \ref{0061}
for Statement \eqref{0053}.

In particular, the conjecture is true
for a proper smooth curve $S$ over a finite field $\F_{q}$.
Hence we have the quasi-compact N\'eron model $G$
of a connected wound unipotent group $G_{K}$
over a function field $K$ over $\F_{q}$
with no non-zero unirational subgroup.
As an application of Statement \eqref{0052},
we can now apply Oesterl\'e's results
in \cite[Chapter VI, Sections 2 and 6]{Oes84} to $G$,
which show that $G_{K}(K)$ is finite%
\footnote{
	Rosengarten has informed the authors that
	he also has obtained the the finiteness of $G_{K}(K)$
	by different methods (unpublished).
}
(\cite[Chapter VI, Question 3.5]{Oes84})
and express the Tamagawa number of $G_{K}$
in terms of the coherent Euler characteristic of the Lie algebra of the N\'eron model $G$
(\cite[Chapter VI, Corollary 6.5]{Oes84}).

By Statement \eqref{0053},
Jacobians stand out as those having N\'eron models.

Conjecture \ref{0012} \eqref{0048} says nothing about
wound unipotent unirational groups,
for which we obtain the following classification
``up to relative perfection''
in the perfect residue field case.
To explain this, following Kato \cite{Kat86},
we say that a scheme $X$
over positive characteristic $K$ is
\emph{relatively perfect}
if the relative Frobenius morphism
$X \to X^{(p)}$ over $K$ is an isomorphism.
If $[K : K^{p}] < \infty$, the inclusion functor
from the category of relatively perfect $K$-schemes
to all $K$-schemes admits a right adjoint,
called the \emph{relative perfection} functor
(\cite[Definition 1.8]{Kat86}).

\begin{theorem}[{Theorem \ref{0049}}] \label{0051}
	Let $K$ be a field of characteristic $p > 0$
	with $[K : K^{p}] = p$.
	Then the category of relative perfections of
	wound unipotent unirational groups over $K$
	with $K$-group scheme morphisms is
	canonically equivalent to
	the category of $p$-primary finite \'etale
	group schemes over $K$.
	Under this equivalence,
	Oesterl\'e's example $(\Res_{K^{1 / p^{n}} / K} \Gm) / \Gm$
	for each $n \ge 1$ corresponds to $\Z / p^{n} \Z$.
\end{theorem}

During the proof, we also prove that for such $K$,
an extension of a unipotent wound unirational group
by a unipotent wound unirational group is unirational
(Proposition \ref{0007}),
partially answering a question of Achet
\cite[Question 4.9]{Ach19a}.
The authors have been informed that this final statement
has also been obtained by Rosengarten.
As an application of Theorem \ref{0051},
we show that the converse to
Conjecture \ref{0012} \eqref{0048}
holds for such $K$
when $S$ has infinitely many points
(Proposition \ref{0064}).
A similar statement is proved
in \cite[Section 10.3, Theorem 5 (b)]{BLR90}
for $S$ a normal curve over a field.

%%%%%%%%%%%%%%%%%%%%%%%%%%%%%%%%%%%%%%%%%%%%%%%%%%%%%%%%%%%%%%%%%%%%%%%%%%%%%

\subsection{Outline of proof}

Throughout,
we heavily rely on relatively perfect
unipotent group schemes and their duality theory
as developed by Kato \cite{Kat86},
Bertapelle and the second author \cite[Section 8]{BS24}
and the second author \cite[Section 3]{Suz24}.

For Theorem \ref{0050} \eqref{0052},
it is relatively straightforward to prove that
the ``dual'' of a connected wound
unipotent group over $K$ killed by $p$
admits a quasi-compact N\'eron model
(Proposition \ref{0003}).
By double duality (\cite[Proposition 3.5]{Suz24}),
this shows that a wound unipotent group
with connected dual killed by $p$
has a quasi-compact N\'eron model.
For these to make sense,
we need to review relatively
perfect schemes in Section \ref{0054},
give general results on them in Section \ref{0055},
define relatively perfect versions of
N\'eron models in Section \ref{0056}
and show that
the existence of a N\'eron model is equivalent to
the existence of a relatively perfect N\'eron model
(Propositions \ref{0062} and \ref{0000}).
Therefore we may pass to relative perfection.
Wound unipotent groups with \'etale dual,
on the other hand, are treated separately
(Proposition \ref{0004}).
It turns out that a wound unipotent group has
no non-zero unirational subgroup
if and only if its dual is connected
(Proposition \ref{0019}).
This needs Achet's result
\cite[Proposition 2.5]{Ach19b}
on coverings of unirational groups
by groups of the form
$(\Res_{K^{1 / p^{n}} / K} \Gm) / \Gm$.
The rest is a d\'evissage
(Sections \ref{0057} and \ref{0058}).

The ``dual'' in the above sense is given by
the sheaf-Hom to the logarithmic Hodge-Witt sheaf
$\nu_{n}(r)$, where $[K : K^{p}] = p^{r}$.
This target sheaf $\nu_{n}(r)$ is isomorphic to
$(\Res_{K^{1 / p^{n}} / K} \Gm) / \Gm$
if $S$ has perfect residue fields or,
equivalently, $r = 1$.
For Theorem \ref{0050} \eqref{0053},
it turns out that $\nu(r) = \nu_{1}(r)$
with $r > 1$ gives a desired counterexample
to Conjecture \ref{0012} \eqref{0047}.
The point is that
special fibers at \emph{all} closed points $s$ of
local N\'eron models of $\nu(r)_{K}$ have
wound quotients given by $\nu(r - 1)_{s}$
while those of $\nu(r)_{S}$ are
a product of copies of $\Ga$.

For Theorem \ref{0051},
we will prove that wound unipotent
unirational groups over $K$
are precisely those with \'etale dual
(Proposition \ref{0007}).
Since the dual is given by sheaf-Hom
to $\nu_{n}(1)$ (where $r = 1$)
or sheaf-$\Ext^{1}$ to $\Gm$,
this is basically Achet's result
\cite[Theorem 2.8]{Ach19b}
on finiteness of $\Ext^{1}_{K}(G_{K}, \Gm)$
for unirational $G_{K}$.
However, passing to duals requires
passing to relative perfections,
which are built from Weil restrictions
$\Res_{K / K} G_{K}$
along the absolute Frobenius morphism
$\Spec K \to \Spec K$.
It is then non-trivial how unirationality behaves
under Weil restriction $\Res_{K / K}$.
In Section \ref{0059},
we study the kernel of the natural morphism
from $\Res_{K / K} G_{K}$ to $G_{K}$.
Together with Rosengarten's result
$\Ext^{2}_{K}(\Ga, \Gm) \cong \Omega^{1}_{K}$
(\cite[Proposition 2.7.6]{Ros23}),
we can now pass to the relative perfection
(Proposition \ref{0045}),
which proves the theorem.

%%%%%%%%%%%%%%%%%%%%%%%%%%%%%%%%%%%%%%%%%%%%%%%%%%%%%%%%%%%%%%%%%%%%%%%%%%%%%

\subsection{Conventions}

All group schemes in this paper are
assumed commutative.
An algebraic group over a field means
a quasi-compact smooth group scheme,
not necessarily connected.

\begin{acknowledgments}
	The authors thank Zev Rosengarten
	for informing them of his work on
    the extension problem for unirational groups
    and Iku Nakamura
    for giving comments for an earlier version of the paper.
    The authors are also grateful to the referees
    for the detailed comments,
    which, in particular, pointed out a gap in a reduction step
    for the connected dual case
    of the proof of Proposition \ref{0006}
    in an earlier manuscript of this paper.
    The first author's research was conducted
    in the framework of the research training group
    \emph{GRK 2240: Algebro-Geometric Methods in Algebra,
    Arithmetic and Topology}.
\end{acknowledgments}

%%%%%%%%%%%%%%%%%%%%%%%%%%%%%%%%%%%%%%%%%%%%%%%%%%%%%%%%%%%%%%%%%%%%%%%%%%%%%

\section{Review of relatively perfect schemes}
\label{0054}

Let $p$ be a prime number.
Let $S$ be an excellent irreducible quasi-compact
regular $\F_{p}$-scheme of dimension $\le 1$
whose residue fields $k(s)$ at closed points $s$
have finite $[k(s) : k(s)^{p}]$.
Let $K$ be the function field of $S$.
Then $[k(s) : k(s)^{p}]$ is independent
of the choice of $s$.
Write it as $p^{r - 1}$ with $r \ge 1$
if $S \ne \Spec K$
and as $p^{r}$ if $S = \Spec K$.
Then the scheme $S$ Zariski-locally admits
a $p$-basis consisting of $r$ elements.
In particular, $[K : K^{p}] = p^{r}$.
Hence the assumptions \cite[(1.4), (3.1.1)]{Kat86}
are satisfied by $S$ and $K$.
See \cite[Section 2.1]{Ove25} for some of these facts.

Recall from \cite[Section 1]{Kat86}
and \cite[Sections 2]{BS24}
that an $S$-scheme $X$ is said to be
\emph{relatively perfect}
if its relative Frobenius morphism
$X \to X^{(p)}$ over $S$ is an isomorphism.
The inclusion functor from
the category of relatively perfect $S$-schemes
into the category of all $S$-schemes
admits a right adjoint,
called the relative perfection functor and denoted by
$X \mapsto X^{\RP}$.
We have $X^{\RP}(S) \cong X(S)$.

We recall the construction of
the relative perfection functor.
Let $n \ge 0$ be an integer.
Let $X^{(p^{n})}$ be the base change of $X$
along the $p^{n}$-th power
absolute Frobenius morphism $S \to S$.
Let $X^{(1 / p^{n})}$ be the Weil restriction of $X$
along the $p^{n}$-th power
absolute Frobenius morphism $S \to S$.
The functor $X \mapsto X^{(1 / p^{n})}$ is
right adjoint to the functor $X \mapsto X^{(p^{n})}$
as endofunctors on the category of $S$-schemes.
We have the relative Frobenius morphism
	\[
			X^{(1 / p^{n + 1})}
		\to
			(X^{(1 / p^{n + 1})})^{(p)}
	\]
of $X^{(1 / p^{n + 1})}$ over $S$.
The counit of adjunction gives a morphism
	\[
			(X^{(1 / p^{n + 1})})^{(p)}
		\cong
			((X^{(1 / p^{n})})^{(1 / p)})^{(p)}
		\to
			X^{(1 / p^{n})}.
	\]
Their composite defines a morphism
	\begin{equation} \label{0036}
			X^{(1 / p^{n + 1})}
		\to
			X^{(1 / p^{n})}.
	\end{equation}
Now $X^{\RP}$ is given by
the inverse limit of the system
$\{X^{(1 / p^{n})}\}_{n \ge 0}$.

An $S$-scheme $X$ is said to be
\emph{locally relatively perfectly of finite presentation}
if Zariski-locally on $S$ and $X$,
it is the relative perfection
of an $S$-scheme of finite presentation.
If moreover $X$ is quasi-compact and quasi-separated,
we say that $X$ is
\emph{relatively perfectly of finite presentation}.
An $S$-scheme $X$ is said to be
\emph{relatively perfectly smooth}
if Zariski-locally on $S$ and $X$,
it is the relative perfection of a smooth $S$-scheme.
When $S$ is a point,
these notions have been studied
in \cite[Section 8]{BS24} in detail.

Let $S_{\RP}$ be the category
of relatively perfect $S$-schemes
with $S$-scheme morphisms endowed
with the \'etale topology,
which is a Grothendieck site.
Let $\Ab(S_{\RP})$ be the category of sheaves
of abelian groups on $S_{\RP}$.
Let $\Hom_{S_{\RP}}$ and $\sheafhom_{S_{\RP}}$ be
the Hom and sheaf-Hom functors,
respectively, for $\Ab(S_{\RP})$.
When $S = \Spec K$, we also use the notation
$\Ab(K_{\RP})$, $\Hom_{K_{\RP}}$
and $\sheafhom_{K_{\RP}}$.
Let $S_{\RPS}$ be the category of
relatively perfectly smooth $S$-schemes
with $S$-scheme morphisms
endowed with the \'etale topology.
It is a site since an \'etale scheme
over a relatively perfectly smooth $S$-scheme
is again a relatively perfectly smooth $S$-scheme.
We similarly have $\Ab(S_{\RPS})$,
$\Hom_{S_{\RPS}}$ and $\sheafhom_{S_{\RPS}}$
and their versions for $K$ in place of $S$.

For integers $n \ge 1$ and $q \ge 0$,
let
	$
			\nu_{n}(q)_{S}^{\RP}
		=
			W_{n} \Omega_{S, \log}^{q, \RP}
		\in
			\Ab(S_{\RP})
	$
be the dlog part of the Hodge-Witt sheaf
(\cite[Section (4.1)]{Kat86}).
Set
$\nu_{\infty}(q)_{S}^{\RP} = \dirlim_{n} \nu_{n}(q)_{S}^{\RP}$
and $\nu(q)_{S}^{\RP} = \nu_{1}(q)_{S}^{\RP}$.
By the proof of \cite[Proposition 6.4]{Kat86},
we have canonical exact sequences
	\begin{equation} \label{0015}
		\begin{gathered}
					0
				\to
					\Gm^{\RP}
				\stackrel{p^{n}}{\to}
					\Gm^{\RP}
				\to
					\nu_{n}(1)_{S}^{\RP}
				\to
					0
			\\
					0
				\to
					\Gm^{\RP}
				\to
					\Gm^{\RP} \tensor_{\Z} \Z[1 / p]
				\to
					\nu_{\infty}(1)_{S}^{\RP}
				\to
					0
		\end{gathered}
	\end{equation}
in $\Ab(S_{\RP})$.

The relative perfections of
quasi-compact smooth group schemes over $K$
are studied in \cite[Section 8]{BS24} in detail.
In particular, they form
an abelian subcategory of $\Ab(K_{\RP})$
closed under extensions by \cite[Proposition 8.12]{BS24}.
A wound unipotent group over $K$ is
a quasi-compact smooth unipotent group scheme
(not necessarily connected)
with no non-zero morphism from $\Ga$.
The relative perfection of
a wound unipotent group is called
a relatively perfect wound unipotent group
(\cite[Section 3]{Suz24}).
By the paragraph after \cite[Definition 4.2.3]{Kat86},
the sheaf $\nu_{n}(q)_{K}^{\RP}$ for any $n$ and $q$ belongs to $D_{0}(K_{\RP})$,
the smallest full triangulated subcategory containing $\Ga^{\RP}$
of the derived category of $\Ab(K_{\RP})$.
Hence by \cite[Proposition 3.2]{Suz24},
the sheaf $\nu_{n}(q)_{K}^{\RP}$ is the relative perfection
of a quasi-compact smooth unipotent group scheme over $K$.
It is wound by \cite[Proposition 3.4]{Suz24},
thus a relatively perfect wound unipotent group.

Let $G_{K}$ be a relatively perfect
wound unipotent group over $K$.
Then by \cite[Proposition 3.5]{Suz24},
the sheaf
	\[
			H_{K}
		=
			\sheafhom_{K_{\RP}}(
				G_{K},
				\nu_{\infty}(r)_{K}^{\RP}
			)
	\]
is representable by a relatively perfect
wound unipotent group.
We call it the \emph{dual} of $G_{K}$.
By \cite[Proposition 3.5]{Suz24},
the dual of $H_{K}$ recovers $G_{K}$.
If $G_{K}$ is a wound unipotent group
(in the usual sense, not relatively perfect),
then its dual is defined to be the dual of $G_{K}^{\RP}$.

%%%%%%%%%%%%%%%%%%%%%%%%%%%%%%%%%%%%%%%%%%%%%%%%%%%%%%%%%%%%%%%%%%%%%%%%%%%%%

\section{General results on relatively perfect schemes}
\label{0055}

The category of $K$-schemes
relatively perfectly of finite presentation
is a ``limit'' of the categories of schemes
relatively perfectly of finite presentation
over dense open subschemes of $S$:

\begin{proposition} \label{0022} \mbox{}
	\begin{enumerate}
		\item \label{0023}
			For any $K$-scheme $X_{K}$
			relatively perfectly of finite presentation,
			there exist a dense open subscheme $T \subset S$
			and a $T$-scheme $X_{T}$
			relatively perfectly of finite presentation
			such that $X_{K} \cong X_{T} \times_{T} K$
			as schemes over $K$.
		\item \label{0024}
			Given $S$-schemes $X$ and $Y$
			relatively perfectly of finite presentation
			and a morphism
			$\varphi_{K} \colon X \times_{S} K \to Y \times_{S} K$
			over $K$,
			there exist a dense open subscheme $T \subset S$
			and a morphism
			$\varphi_{T} \colon X \times_{S} T \to Y \times_{S} T$
			whose base change to $K$ is $\varphi_{K}$.
		\item \label{0025}
			Given $S$-schemes $X$ and $Y$
			relatively perfectly of finite presentation
			and a pair of morphisms $\varphi, \psi \colon X \to Y$
			whose base changes
			$\varphi \times_{S} K$ and $\psi \times_{S} K$
			are equal,
			there exists a dense open subscheme $T \subset S$
			such that $\varphi \times_{S} T$ and $\psi \times_{S} T$
			are equal.
	\end{enumerate}
\end{proposition}

(Here the notation $X_{T} \times_{T} K$ means
$X_{T} \times_{T} \Spec K$.
Similar such notations have been/will be used throughout.)

\begin{proof}
	By the argument of the proof of \cite[Tag 01ZM]{Sta26},
	all the statements are reduced to the case
	where $S$ is affine and the schemes are
	relative perfections of affine schemes of finite presentation.
	In this case, the statements are further reduced to
	the usual statements for algebras of finite presentation,
	which are \cite[Tag 05N9]{Sta26}.
\end{proof}

In Proposition \ref{0026} below,
we will construct the ``relative identity component''
of a relatively perfectly smooth group scheme over $S$.
During the proof of its quasi-compactness,
we will need the following technical result,
just as in the case of group schemes over a field
(see the proof of \cite[Tag 0B7P]{Sta26}):

\begin{proposition} \label{0020}
	Let $X$ be a relatively perfectly smooth $S$-scheme.
	Then the structure morphism
	$X \to S$ is universally open.
\end{proposition}

\begin{proof}
	We may assume that $S$ is affine and
	$X$ is the relative perfection
	of an $S$-scheme $X_{0}$ that is \'etale
	over some affine space $\Affine_{S}^{n}$ over $S$.
	Since relative perfection commutes with \'etale morphisms
	by \cite[Corollary 1.9]{Kat86},
	it is enough to show that
	$(\Affine_{S}^{n})^{\RP} \to S$ is
	universally open.
	The inverse system
	$\{(\Affine_{S}^{n})^{(1 / p^{m})}\}_{m \ge 0}$
	has faithfully flat transition morphisms.
	Now use the fact
	(which follows from
	\cite[Tags 01Z4 (1), 05F3 and 01UA]{Sta26})
	that for a filtered inverse system
	$\{Y_{\lambda}\}_{\lambda \in \Lambda}$
	of quasi-compact quasi-separated schemes
	whose transition morphisms
	$Y_{\mu} \to Y_{\lambda}$ are
	faithfully flat affine of finite presentation,
	the morphism
	$\invlim_{\mu} Y_{\mu} \to Y_{\lambda}$
	is universally open for any $\lambda \in \Lambda$.
\end{proof}

As mentioned in the proof above,
the relative perfection functor commutes
with \'etale base change.
But it does not in general with closed immersions.
Hence special fibers of relative perfections
of schemes over $S$ can be messy.
However, if we only care about
the irreducibility of the fibers
or the set of irreducible components,
relative perfection does not change the result,
at least for smooth schemes over $S$:

\begin{proposition} \label{0021}
	Let $X$ be a smooth $S$-scheme.
	Let $s \in S$ be a closed point.
	Then the fiber
	$(X^{\RP})_{s} = X^{\RP} \times_{S} s$ has
	open irreducible components and the morphism
	$(X^{\RP})_{s} \to X_{s}$ induces a bijection
	on the sets of irreducible components.
\end{proposition}

\begin{proof}
	First note that $X^{(1 / p^{n})}$ is smooth over $S$ for any $n$
	by \cite[Section 7.6, Proposition 5 (h)]{BLR90}
	and that a smooth scheme over a field has
	open irreducible components by \cite[Tag 033N]{Sta26}.
	
	Now, to prove the proposition, it is enough to show that
	$(X^{(1 / p)})_{s} \to X_{s}$ induces a bijection
	on the sets of irreducible components.
	Indeed, this implies,
	by applying it to $X^{(1 / p^{n})}$ in place of $X$ for each $n$,
	that the same is true
	for $(X^{(1 / p^{n + 1})})_{s} \to (X^{(1 / p^{n})})_{s}$.
	From this, we obtain the proposition by taking the limit in $n$,
	since a disjoint open decomposition of $X_{s}$ induces
	a disjoint open decomposition of the limit $(X^{\RP})_{s}$
	and a filtered inverse limit of quasi-separated irreducible schemes
	with affine transition morphisms
	is irreducible.%
	\footnote{
		To see this last statement,
		let $Y = \invlim_{\lambda} Y_{\lambda}$ be such an inverse limit.
		The underlying topological space of $Y$ is the inverse limit of
		the underlying topological spaces of the $Y_{\lambda}$
		by \cite[Tag 0CUF]{Sta26}.
		Hence $Y$ has an open base given by $U_{\lambda} \times_{Y_{\lambda}} Y$,
		where $\lambda$ runs over all indices
		and $U_{\lambda}$ runs over all quasi-compact open subschemes of $Y_{\lambda}$
		with $U_{\lambda} \times_{Y_{\lambda}} Y \ne \emptyset$.
		It is enough to show that any two members
		$U_{\lambda} \times_{Y_{\lambda}} Y$ and $V_{\mu} \times_{Y_{\mu}} Y$
		of this base have a non-empty intersection.
		We may assume that $\lambda = \mu$.
		For any $\nu \ge \lambda$,
		the quasi-separatedness and the irreducibility of $Y_{\nu}$ imply that
			$
					(U_{\lambda} \cap V_{\lambda}) \times_{Y_{\lambda}} Y_{\nu}
				\cong
						(U_{\lambda} \times_{Y_{\lambda}} Y_{\nu})
					\cap
						(V_{\lambda} \times_{Y_{\lambda}} Y_{\nu})
			$
		is quasi-compact and non-empty.
		Hence by \cite[Tag 01Z2]{Sta26},
		we know that
			$
					(U_{\lambda} \cap V_{\lambda}) \times_{Y_{\lambda}} Y
				\cong
						(U_{\lambda} \times_{Y_{\lambda}} Y)
					\cap
						(V_{\lambda} \times_{Y_{\lambda}} Y)
			$
		is non-empty.
	}
	
	To prove this statement about $(X^{(1 / p)})_{s} \to X_{s}$,
	note that the statement is local on $X$ and a neighborhood of $s$,
	so we may assume that
	$X$ is \'etale over an affine space $\Affine_{S}^{n}$.
	The relative Frobenius
	$X^{(1 / p)} \to (X^{(1 / p)})^{(p)}$
	is a universal homeomorphism.
	By an explicit calculation,
	the counit
	$((\Affine_{S}^{n})^{(1 / p)})^{(p)} \to \Affine_{S}^{n}$
	is a linear projection of affine spaces.
	It follows that the counit $(X^{(1 / p)})^{(p)} \to X$ is
	isomorphic to the projection morphism
	$\Affine_{X}^{m} \to X$ for some $m$,
	which induces a bijection on the sets of irreducible components
	by \cite[Tag 037A or 038F]{Sta26}.
\end{proof}

Now we construct relative identity components:

\begin{proposition} \label{0026}
	Let $G$ be a relatively perfectly smooth
	group scheme over $S$.
	Then there exists a unique open subscheme $G^{0}$ of $G$
	whose fibers are the identity components
	of the fibers of $G$.
	The scheme $G^{0}$ is
	a quasi-compact group subscheme of $G$.
	We call $G^{0}$
	the \emph{relative identity component} of $G$.
\end{proposition}

\begin{proof}
	Let $\pi \colon G \to S$ be the structure morphism.
	Let $e \colon S \into G$ be the identity section.
	By assumption, $S \cong e(S)$ is quasi-compact.
	Cover $e(S)$ by finitely many open subschemes
	$U_{i} = (U_{i, 0})^{\RP} \subset G$
	each of which is the relative perfection
	of a quasi-compact smooth $S$-scheme $U_{i, 0}$
	such that $e(\pi(U_{i})) \subset U_{i}$.
	Let $T_{i} = e^{-1}(U_{i})$.
	The morphisms $\pi$ and $e$ induce a morphism
	$\pi \colon U_{i, 0} \to T_{i}$ and
	its section $e \colon T_{i} \to U_{i, 0}$.
	Then by \cite[Tag 055R]{Sta26},
	there exists a (unique) open subscheme
	$V_{i, 0} \subset U_{i, 0}$
	whose fiber over any $s \in T_{i}$ is
	the connected component of the fiber $(U_{i, 0})_{s}$
	containing $e(s)$.
	Let $V = \bigcup_{i} (V_{i, 0})^{\RP}$.
	It is a quasi-compact open subscheme of $G$
	whose fiber over any $s \in S$ is
	a dense open subscheme of $(G_{s})^{0}$
	by Proposition \ref{0021}.
	By Proposition \ref{0020},
	the first projection $G \times_{S} G \to G$ is open.
	Hence the morphism
	$\varphi \colon G \times_{S} G \to G$,
	$(g, h) \mapsto g -  h$
	is open.
	Therefore the composite
	$V \times_{S} V \into G \times_{S} G \to G$
	of the inclusion and $\varphi$ is open.
	The image of $V \times_{S} V \to G$
	is quasi-compact and
	its fiber over any $s \in S$ is $(G_{s})^{0}$
        by the argument from the proof of \cite[Tag 0B7P]{Sta26}.
	Hence this image gives the desired object $G^{0}$.
\end{proof}

%%%%%%%%%%%%%%%%%%%%%%%%%%%%%%%%%%%%%%%%%%%%%%%%%%%%%%%%%%%%%%%%%%%%%%%%%%%%%

\section{Relatively perfect N\'eron models}
\label{0056}

The base change functor
$X \times_{S} K \mapsfrom X$
defines a morphism of sites
    \[
        j \colon \Spec K_{\RPS} \to S_{\RPS}.
    \]

\begin{definition} \label{0046}
	Let $X_{K}$ be a relatively perfectly smooth $K$-scheme.
	A relatively perfectly smooth $S$-scheme $X$
	together with an isomorphism
	$X \times_{S} K \cong X_{K}$
	is called a \emph{relatively perfect N\'eron model}
	(over $S$) of $X_{K}$
	if for any relatively perfectly smooth $S$-scheme $Y$,
	any morphism $Y \times_{S} K \to X_{K}$ over $K$
	uniquely extends to a morphism $Y \to X$ over $S$.
\end{definition}

In other words,
a relatively perfect N\'eron model $X$ of $X_{K}$, if exists,
is the object representing
the sheaf of sets $j_{\ast} X_{K}$.
If $X_K$ is a group scheme over $K$, then $X$ is a group scheme over $S$ if it exists. Note that, in particular, the scheme $X$ is separated in this case. Indeed, an argument as in the proof of \cite[Section 7.1, Theorem 1, (iii) $\Rightarrow$ (ii)]{BLR90} shows that the unit section $S\to X$ is a closed immersion, so this follows from \cite[Section 7.1, Lemma 2]{BLR90}.
To check if $X$ is a relatively perfect N\'eron model,
it suffices to check that Zariski-locally on $S$,
the above mapping property holds for
all $Y$ which are relative perfections of smooth affine $S$-schemes.

We will show in
Propositions \ref{0062} and \ref{0000} below
that the existence of a N\'eron model
and the existence of a relatively perfect N\'eron model
are essentially equivalent.

\begin{proposition} \label{0062}
	Let $G_{K}$ be a smooth group scheme over $K$.
	Assume that $G_{K}$ admits a N\'eron model $G$ over $S$.
	Then $G_{K}^{\RP}$ admits a relatively perfect
	N\'eron model over $S$
	given by $G^{\RP}$.
\end{proposition}

\begin{proof}
	Let $Y$ be a smooth affine $S$-scheme
	and $Y^{\RP} \times_{S} K \to G_{K}$
	a $K$-scheme morphism.
	It factors through some morphism
	$Y^{(1 / p^{n})} \times_{S} K \to G_{K}$
	for some $n$.
	Since $Y^{(1 / p^{n})}$ is smooth over $S$,
	this morphism extends to a morphism
	$Y^{(1 / p^{n})} \to G$ over $S$
	by assumption.
        The composite $Y^{\RP} \to Y^{(1/{p^n})} \to G$
        gives the desired extension.
\end{proof}

Before giving the converse,
we need to show a criterion
for existence of local N\'eron models,
which is a relatively perfect version of
the criterion for local existence
\cite[Section 10.2, Theorem 2]{BLR90}.
The proof of the key proposition
\cite[Section 10.1, Proposition 8]{BLR90}
does not translate to the relatively perfect setting
as relative perfection does not commute with
the formation of special fibers.
Instead, we will use
P\'epin's N\'eron ind-scheme models
(\cite[Sections 5.1 and 5.2]{Pep14},
\cite[Theorem 3.2]{Pep15}).
It is convenient for us
to slightly generalize his existence result
(\cite[Theorem 5.2.8]{Pep14})
by dropping the completeness
or separable closedness,
so we (essentially) recall
some part of his proof below.

We need some notation.
Let $S_{\sm}$ be
the category of smooth $S$-schemes
with $S$-scheme morphisms
endowed with the \'etale topology.
Let $\Spec K_{\sm}$ be defined similarly.
Let
	\[
			j_{0}
		\colon
			\Spec K_{\sm}
		\to
			S_{\sm}
	\]
be the morphism of sites defined by the functor
$X \times_{S} K \mapsfrom X$
on the underlying categories.
Let $\Ab(S_{\sm})$ be
the category of sheaves of abelian groups
on $S_{\sm}$.

For a connected smooth group scheme
$G_{K, 0}$ over $K$,
let $G_{K, 0}' \subset G_{K, 0}$ be
the maximal smooth connected split
unipotent subgroup
(\cite[Corollary B.3.5]{CGP15}).
Let $G_{K, 0}'' = G_{K, 0} / G_{K, 0}'$.
We have an exact sequence
	\begin{equation} \label{0070}
			0
		\to
			G_{K, 0}'
		\to
			G_{K, 0}
		\to
			G_{K, 0}''
		\to
			0.
	\end{equation}
The group $G_{K, 0}''$ does not contain $\Ga$.
Hence if $S$ is local,
then $G_{K, 0}''$ admits a N\'eron model $G_{0}''$
by \cite[Section 10.2, Theorem 2]{BLR90} and the excellence of $S$,
and applying $j_{0, \ast}$ to \eqref{0070}
yields an exact sequence
	\begin{equation} \label{0069}
			0
		\to
			j_{0, \ast} G_{K, 0}'
		\to
			j_{0, \ast} G_{K, 0}
		\to
			G_{0}''
		\to
			0
	\end{equation}
in $\Ab(S_{\sm})$
since $R^{1} j_{0, \ast} \Ga = 0$.

\begin{proposition} \label{0068}
	Assume that $S$ is local.
	Let $G_{K, 0}$, $G_{K, 0}'$,
	$G_{K, 0}''$ and $G_{0}''$ be
	as above.
	Then the sequence \eqref{0069}
	can be written as the direct limit
	of a system
		\[
				0
			\to
				G_{n, 0}'
			\to
				G_{n, 0}
			\to
				G_{0}''
			\to
				0
		\]
	of exact sequences indexed by $\N$ of
	smooth separated group schemes over $S$
	such that the transition morphisms for $G''_{0}$ are
	the identity morphisms,
	the other transition morphisms
	$G_{n, 0}' \to G_{n + 1, 0}'$
	and $G_{n, 0} \to G_{n + 1, 0}$
	when base-changed to $K$ are isomorphisms
	and each $G_{n, 0}'$ has a finite filtration
	whose successive subquotients are
	all isomorphic to $\Ga$.
\end{proposition}

\begin{proof}
	Write $S = \Spec R$.
	Let $\pi$ be a prime element of $R$.
	Note that $j_{0, \ast} \Ga$ is given by
	the direct limit of the system
	$\Ga \to \Ga \to \cdots$ of
	multiplication by $\pi$ on $\Ga$
	(\cite[Lemma 5.2.2]{Pep14}).
	More generally,
	the sheaf $j_{0, \ast} G_{K, 0}'$ can be written as
	the direct limit of a system
	$\{G_{n, 0}'\}_{n \ge 0}$,
	where each $G_{n, 0}'$ is
	a smooth separated group scheme over $S$
	that has a finite filtration
	whose successive subquotients are
	all isomorphic to $\Ga$
	and the transition morphisms are isomorphisms
	when base-changed to $K$
	(\cite[Lemma 5.2.4]{Pep14}).
	
	With this and \eqref{0069},
	to prove the proposition,
	it is then enough to prove that
	the Ext functor
	$\Ext^{m}_{S_{\sm}}(G_{0}'', \;\cdot\;)$
	on $\Ab(S_{\sm})$ for any $m$ commutes
	with filtered direct limits.
	Consider the relative connected-\'etale sequence
		\[
				0
			\to
				G_{0}''^{0}
			\to
				G''_{0}
			\to
				i_{\ast}
				\pi_{0}(G''_{0, s})
			\to
				0.
		\]
	Since $G_{0}''^{0}$ is quasi-compact,
	the \'etale cohomology functor
	$H^{m}(G_{0}''^{0}, \;\cdot\;)$
	commutes with filtered direct limits
	(\cite[Tag 0739]{Sta26}).
	The same is true for any finite product
	of copies of $G_{0}''^{0}$.
	It follows by a standard
	Deligne-Scholze (or Breen-Deligne)
	resolution argument that
	$\Ext^{m}_{S_{\sm}}(G_{0}''^{0}, \;\cdot\;)$
	also commutes with filtered direct limits
	(see the proof of \cite[Proposition 8.10]{BS24}
	for example).%
	\footnote{
		This argument also shows that
		$\Ext^{m}_{S_{\sm}}(\Ga, \;\cdot\;)$
		commutes with filtered direct limits,
		with which the reader can reconstruct
		the proof of \cite[Lemma 5.2.4]{Pep14} cited above.
	}
	On the other hand,
	the exact sequence $0 \to j_{0, !} \Z \to \Z \to i_{\ast} \Z \to 0$
	induces a long exact sequence
		\[
				\cdots
			\to
				\Ext^{m}_{S_{\sm}}(i_{\ast} \Z, \;\cdot\;)
			\to
				H^{m}(S, \;\cdot\;)
			\to
				H^{m}(K, j_{0}^{\ast} \;\cdot\;)
			\to
				\cdots.
		\]
	This implies that $\Ext^{m}_{S_{\sm}}(i_{\ast} \Z, \;\cdot\;)$
	commutes with filtered direct limits.
	Since the component group
	$\pi_{0}(G_{0, s}'')$
	has finitely generated group of geometric points
	by \cite[Proposition 3.5]{HN11},
	it follows that
	$\Ext^{m}_{S_{\sm}}(i_{\ast} \pi_{0}(G''_{0, s}), \;\cdot\;)$
	also commutes with filtered direct limits.
	Thus $\Ext^{m}_{S_{\sm}}(G_{0}'', \;\cdot\;)$
	commutes with filtered direct limits.
\end{proof}

Now we can prove the following
criterion for local existence
of relatively perfect N\'eron models:

\begin{proposition} \label{0063}
	Assume that $S$ is local.
	Let $G_{K}$ be a connected relatively perfectly
	smooth group scheme over $K$.
	Then $G_{K}$ admits a relatively perfect
	N\'eron model over $S$
	if and only if it does not contain $\Ga^{\RP}$
	as a closed subgroup scheme.
\end{proposition}

\begin{proof}
	Choose a connected smooth group scheme
	$G_{K, 0}$ over $K$
	whose relative perfection is $G_{K}$
	(\cite[Propositions 8.7 and 8.13]{BS24}).
	
	Assume that $G_{K}$ does not contain $\Ga^{\RP}$.
	Then $G_{K, 0}$ does not contain $\Ga$.
	Hence it admits a N\'eron model
	by \cite[Section 10.2, Theorem 2]{BLR90}.
	This implies that $G_{K}$ admits
	a relatively perfect N\'eron model
	by Proposition \ref{0062}.
	
	Next, assume that
	$G_{K}$ admits a relatively perfect
	N\'eron model $G$ over $S$.
	We apply the statement and the notation of
	Proposition \ref{0068} to $G_{K, 0}$.
	This gives groups
	$G_{n, 0}'$, $G_{n, 0}$, $G_{0}''$
	and so on.
	Let $G_{K}'$, $G_{K}''$,
	$G_{n}'$, $G_{n}$, $G''$ be
	the relative perfections of
	$G_{K, 0}'$, $G_{K, 0}''$,
	$G_{n, 0}'$, $G_{n, 0}$, $G_{0}''$,
	respectively.
	
	We show that $G_{K}'$ admits
	a quasi-compact relatively perfect
	N\'eron model.
	The group $G$ is the direct limit
	of the system $\{G_{n}\}_{n}$
	by (the ind analogue of) Proposition \ref{0062}.
	Let $i_{n} \colon G_{n} \to G$ be the natural morphism.
	Then the relative identity component $G^{0}$
	is the direct limit of the system
	$\{i_{n}^{-1} G^{0}\}_{n}$ of inverse images.
	We have $G'_{n} \subset i_{n}^{-1} G^{0}$.
	The transition morphisms
	$G_{n, 0} \to G_{n + 1, 0}$
	are injective in $\Ab(S_{\sm})$
	due to separatedness.
	Since $G^{0}$ is quasi-compact
	(Proposition \ref{0026}),
	it follows that the transition morphisms
	$i_{n}^{-1} G^{0} \to i_{n + 1}^{-1} G^{0}$
	are isomorphisms for all large $n \ge N$
        by \cite[Tag 0738]{Sta26}.
	This implies that
	$G_{n}' \to G_{n + 1}'$ is an isomorphism for $n \ge N$.
	Denote these isomorphic groups as $G'$.
	Then $G'$ gives a relatively perfect
	N\'eron model of $G'_{K}$.
	It is quasi-compact
	since each $G_{n}'$ is.
	
	Now, suppose for contradiction that
	$G_{K}$ contains $\Ga^{\RP}$.
	Then $G_{K}'$ contains $\Ga^{\RP}$
	since $G_{K, 0}' \subset G_{K, 0}$ is the maximal smooth
	connected split unipotent subgroup.
	In particular, $G_{K}' \ne 0$.
	Hence we have an exact sequence
		\[
				0
			\to
				G_{K}'''
			\to
				G_{K}'
			\to
				\Ga^{\RP}
			\to
				0,
		\]
	where $G_{K}'''$ has a finite filtration
	whose successive subquotients are
	all isomorphic to $\Ga^{\RP}$.
	Hence we have an exact sequence
		\[
				0
			\to
				j_{\ast} G_{K}'''
			\to
				G'
			\to
				j_{\ast} \Ga^{\RP}
			\to
				0
		\]
	in $\Ab(S_{\RPS})$
        because $R^{1} j_{\ast} \Ga^{\RP}=0$.
	The sheaf $j_{\ast} \Ga^{\RP}$ is given by
	the direct limit of the system
	$\Ga^{\RP} \to \Ga^{\RP} \to \cdots$ of
	multiplication by $\pi$ on $\Ga^{\RP}$.
	The quasi-compactness of $G'$ implies that
	the morphism $G' \onto j_{\ast} \Ga^{\RP}$
	factors through the $n$-th term $\Ga^{\RP}$
	for some $n$ \cite[Tag 0738]{Sta26}.
	This implies that
	the $n$-th morphism
	$\Ga^{\RP} \to j_{\ast} \Ga^{\RP}$
	is surjective in $\Ab(S_{\RPS})$.
	Hence the morphism $R^{sh} \to K^{sh}$
	given by multiplication by $\pi^{- n}$
	is surjective,
	which is a contradiction.
	Therefore $G_{K}$ does not contain $\Ga^{\RP}$.
\end{proof}

Here is the converse to Proposition \ref{0062}:

\begin{proposition} \label{0000}
	Let $G_{K}$ be a smooth group scheme over $K$.
	Assume that $G_{K}^{\RP}$ admits
	a relatively perfect N\'eron model over $S$.
	Then $G_{K}$ admits a N\'eron model over $S$.
\end{proposition}

\begin{proof}
	If $G_{K}$ is \'etale
	with finitely generated group of geometric points,
	then it admits an \'etale N\'eron model over $S$
	by \cite[Section 10.1, Proposition 4]{BLR90}.
	The finite generation assumption can be dropped
	by a limit argument.
	
	In general, let $G$ be the relatively perfect
	N\'eron model of $G_{K}^{\RP}$
	and $N$ the N\'eron model of $\pi_{0}(G_{K})$.
	Since the identity section $S \into N$ of $N$ is an open immersion,
	the kernel of $G \to N$ is an open subgroup scheme of $G$
	and, in particular, relatively perfectly smooth over $S$.
	Hence this kernel gives a relatively perfect
	N\'eron model of $(G_{K}^{0})^{\RP}$.
	With \cite[Section 7.5, Proposition 1]{BLR90}
	(which generalizes to non-quasi-compact N\'eron models
	as remarked after
	\cite[Section 10.3, Conjecture II]{BLR90};
	see also \cite[Corollary 2.6]{Ove24}),
	we may thus assume that $G_{K}$ is connected.
	
	By Proposition \ref{0063},
	$G_{K}^{\RP}$ does not contain $\Ga^{\RP}$
	and hence $G_{K}$ does not contain $\Ga$.
	Therefore $G_{K}$ admits a N\'eron model
	at each closed point of $S$
	by \cite[Section 10.2, Theorem 2]{BLR90}.
	
	Let $G$ be a relatively perfect
	N\'eron model of $G_{K}^{\RP}$.
	Let $G^{0} \subset G$ be its relative identity component
	(Proposition \ref{0026}),
	which is quasi-compact.
	By spreading out $G_{K}$,
	choose a smooth separated group scheme $G_{0}$
	over a dense open subscheme $U \subset S$
	with connected fibers and generic fiber $G_{K}$.
	We have a natural morphism
	$G_{0}^{\RP} \to G^{0}$ over $U$.
	Since it is an isomorphism over $K$,
	by Proposition \ref{0022} \eqref{0024}
	and \eqref{0025},
	we may assume that
	it is an isomorphism by shrinking $U$ if necessary.
	
	For each closed point $s \in S$,
	let $G_{\Order_{S, s}}$ be the N\'eron model
	over $\Order_{S, s}$ of $G_{K}$.
	Then we have a natural morphism
	$G_{0} \times_{S} \Order_{S, s} \to G_{\Order_{S, s}}^{0}$
	over $\Order_{S, s}$
	and hence a map
		$
				G_{0}(\Order_{S, s}^{sh})
			\to
				G_{\Order_{S, s}}^{0}(\Order_{S, s}^{sh})
		$.
	Since
	$G_{\Order_{S, s}}^{\RP} \cong G \times_{S} \Order_{S, s}$,
	we have
	$G_{\Order_{S, s}}^{0, \RP} \cong G^{0} \times_{S} \Order_{S, s}$.
	As $\Spec \Order_{S, s}^{sh}$ is
	relatively perfect over $S$,
	this implies
		$
				G_{\Order_{S, s}}^{0}(\Order_{S, s}^{sh})
			\cong
				G^{0}(\Order_{S, s}^{sh})
		$.
	Since $G_{0}^{\RP} \isomto G^{0}$, this implies that
	the map
		$
				G_{0}(\Order_{S, s}^{sh})
			\to
				G_{\Order_{S, s}}^{0}(\Order_{S, s}^{sh})
		$
	is an isomorphism.
	
	In particular, the map
	$G_{0}(k(s)^{\sep}) \to G_{\Order_{S, s}}^{0}(k(s)^{\sep})$
	is surjective.
	Hence
	$G_{0} \times_{S} \Order_{S, s} \to G_{\Order_{S, s}}^{0}$
	is an isogeny and, in particular, flat.
	Since it is an isomorphism on the generic fibers,
	it follows that
	$G_{0} \times_{S} \Order_{S, s} \isomto G_{\Order_{S, s}}^{0}$.
	This implies that $G_{K}$ admits a N\'eron model over $S$
	by \cite[Section 10.1, Proposition 9]{BLR90}.
\end{proof}

By \cite[Section 7.5, Proposition 1]{BLR90},
the existence of N\'eron models
inherits to closed subgroups and extensions.
It has the relatively perfect version:

\begin{proposition} \label{0018}
	Let $0 \to G_{K}' \to G_{K} \to G_{K}'' \to 0$ be
	an exact sequence of quasi-compact
	relatively perfectly smooth group schemes over $K$.
	\begin{enumerate}
		\item \label{0016}
			If $G_{K}$ admits a relatively perfect
			N\'eron model, then so does $G_{K}'$.
		\item \label{0017}
			If $G_{K}'$ and $G_{K}''$ admit
			relatively perfect N\'eron models,
			then so does $G_{K}$.
	\end{enumerate}
	The same are true
	with ``relatively perfect N\'eron models''
	replaced by
	``quasi-compact relatively perfect N\'eron models''.
\end{proposition}

\begin{proof}
	From the proof of \cite[Proposition 8.12]{BS24},
	it follows that there exists an exact sequence
	$0 \to G_{K, 0}' \to G_{K, 0} \to G_{K, 0}'' \to 0$
	of quasi-compact smooth group schemes over $K$
	whose relative perfection recovers the sequence
	$0 \to G_{K}' \to G_{K} \to G_{K}'' \to 0$.
	Hence by Propositions \ref{0062}, \ref{0063} and \ref{0000},
	the statements reduce to
	\cite[Section 7.5, Proposition 1]{BLR90}
	(which generalizes to non-quasi-compact N\'eron models
	as mentioned in the proof of Proposition \ref{0000}).
\end{proof}

%%%%%%%%%%%%%%%%%%%%%%%%%%%%%%%%%%%%%%%%%%%%%%%%%%%%%%%%%%%%%%%%%%%%%%%%%%%%%

\section{Existence of N\'eron models: fundamental cases}

Assume that $S \ne \Spec K$ and $r = 1$,
namely that $S$ has perfect residue fields
at closed points.
In this section, we will construct N\'eron models
in two fundamental cases.

We first need the following
relatively perfect version
of Hilbert's theorem 90:

\begin{proposition} \label{0065}
	We have $R^{1} j_{\ast} \Gm^{\RP} = 0$.
\end{proposition}

\begin{proof}
	It is enough to show that
	$H^{1}((X_{x}^{sh})_{K}, \Gm^{\RP}) = 0$
	for the generic fiber $(X_{x}^{sh})_{K}$
	of the strict localization $X_{x}^{sh}$
	at a point $x$
	of the relative perfection $X$
	of an affine smooth $S$-scheme $X_{0}$.
	Let $x_{n}$ be the image of $x$
	under the morphism $X \to X_{0}^{(1 / p^{n})}$.
	Then $X_{x}^{sh}$ is isomorphic to
	the inverse limit of
	$(X_{0}^{(1 / p^{n})})_{x_{n}}^{sh}$
	in $n$.
	Hence
		\begin{align*}
					H^{1}((X_{x}^{sh})_{K}, \Gm^{\RP})
			&	\cong
					H^{1}((X_{x}^{sh})_{K}, \Gm)
			\\
			&	\cong
					\dirlim_{n}
						H^{1} \bigl(
							((X_{0}^{(1 / p^{n})})_{x_{n}}^{sh})_{K},
							\Gm
						\bigr)
		\end{align*}
	by \cite[Proposition 8.9]{BS24} and its proof.
	For each $n$, the group
		$
			H^{1} \bigl(
				((X_{0}^{(1 / p^{n})})_{x_{n}}^{sh})_{K},
				\Gm
			\bigr)
		$
	is a quotient of
		$
			H^{1}(
				(X_{0}^{(1 / p^{n})})_{x_{n}}^{sh},
				\Gm
			)
		$
	since $(X_{0}^{(1 / p^{n})})_{x_{n}}^{sh}$ is regular,
	and this final group is zero
	since $(X_{0}^{(1 / p^{n})})_{x_{n}}^{sh}$ is local.
\end{proof}

Now we treat the first fundamental case:

\begin{proposition} \label{0003}
	Let $G_{K}$ be a connected wound unipotent group over $K$
	killed by $p$ with connected dual.
	Then it admits a quasi-compact N\'eron model over $S$.
\end{proposition}

To prove this, we use the dual of $G_{K}$ in a crucial manner.

\begin{proof}
	By Proposition \ref{0000},
	it is enough to show that $G_{K}^{\RP}$ admits
	a quasi-compact relatively perfect N\'eron model
	over a dense open subscheme of $S$.
	Let $H_{K}$ be the dual of $G_{K}$.
	Write $H_{K}$ as the relative perfection
	of a quasi-compact smooth group scheme $H_{K, 0}$ over $K$.
	It is wound unipotent, connected and killed by $p$.
	After shrinking $S$ if necessary,
	choose a smooth separated group scheme $H_{0}$
	over $S$ with connected fibers
	and generic fiber $H_{K, 0}$.
	Set $H = H_{0}^{\RP}$.
	We will show that, after shrinking $S$ if necessary,
	the sheaf
	$G := \sheafhom_{S_{\RP}}(H, \nu_{\infty}(1)_{S}^{\RP})$
	is representable by a quasi-compact
	relatively perfectly smooth group scheme over $S$
	and satisfies the relatively perfect
	N\'eron mapping property for $G_{K}$
	when restricted to $S_{\RPS}$
	(Definition \ref{0046}).
	
	Representability:
	By \cite[Section 10.2, Proposition 10]{BLR90},
	there exists an exact sequence
	$0 \to H_{K, 0} \to \Ga^{n} \to \Ga \to 0$ over $K$.
	By spreading out and shrinking $S$ if necessary,
	this sequence extends to an exact sequence
	$0 \to H_{0} \to \Ga^{n} \to \Ga \to 0$ over $S$.
	Hence we have an exact sequence
	$0 \to H \to (\Ga^{\RP})^{n} \to \Ga^{\RP} \to 0$
	over $S$ (where the surjectivity
	of the last morphism follows from
	the same argument as the proof of
	\cite[Proposition 8.8]{BS24}).
	By \cite[Theorem 3.2]{Kat86},
	this induces an exact sequence
		\[
				0
			\to
				G
			\to
				(\Omega_{S}^{1})^{\RP}
			\to
				((\Omega_{S}^{1})^{\RP})^{n}
		\]
	over $S$,
	where the sheaf of absolute differentials
	$\Omega_{S}^{1} = \Omega_{S / \F_{p}}^{1}$ is viewed
	as a vector group over $S$.
	Since the latter two are $S$-schemes
	relatively perfectly of finite presentation,
	it follows that $G$ is representable
	by an $S$-scheme relatively perfectly of finite presentation.
	Shrinking $S$ if necessary,
	by Proposition \ref{0022} (all the statements),
	we can take $G$ to be relatively perfectly smooth.
	
	Mapping property:
	Now we work over $S_{\RPS}$.
	We have $j^{\ast} H \cong H_{K}$.
	Hence
		\[
				j_{\ast} G_{K}
			=
				j_{\ast}
				\sheafhom_{K_{\RPS}}(
					H_{K},
					\nu_{\infty}(1)_{K}^{\RP}
				)
			\cong
				\sheafhom_{S_{\RPS}}(
					H,
					j_{\ast}
					\nu_{\infty}(1)_{K}^{\RP}
				).
		\]
	We have an exact sequence
		\begin{equation} \label{0066}
				0
			\to
				\Gm^{\RP}
			\to
				j_{\ast} \Gm^{\RP}
			\to
				\bigoplus_{s \in S_{0}}
					i_{s, \ast} \Z
			\to
				0,
		\end{equation}
	where $S_{0}$ is the set of closed points of $S$
	and $i_{s, \ast} \Z$ is the \'etale group scheme over $S$
	trivial outside $s$ having $\Z$ as the fiber at $s$.
        We have
            \[
                    (j_{\ast} \Gm^{\RP}) \tensor \Q_{p} / \Z_{p}
                \cong
                    j_{\ast}
                    \nu_{\infty}(1)_{K}^{\RP}
            \]
        by \eqref{0015} and Proposition \ref{0065}
        (plus the fact that $j_{\ast}$ commutes with
        filtered direct limits
        by \cite[Tag 0738]{Sta26}).
	Hence tensoring $\Q_{p} / \Z_{p}$ with \eqref{0066},
        we have an exact sequence
		\begin{equation} \label{0060}
				0
			\to
				\nu_{\infty}(1)_{S}^{\RP}
			\to
				j_{\ast}
				\nu_{\infty}(1)_{K}^{\RP}
			\to
				\bigoplus_{s \in S_{0}}
					i_{s, \ast} \Q_{p} / \Z_{p}
			\to
				0,
		\end{equation}
	with $i_{s, \ast} \Q_{p} / \Z_{p}$ similarly defined.
	It induces an exact sequence
		\[
				0
			\to
				\sheafhom_{S_{\RPS}}(
					H,
					\nu_{\infty}(1)_{S}^{\RP}
				)
			\to
				\sheafhom_{S_{\RPS}}(
					H,
					j_{\ast}
					\nu_{\infty}(1)_{K}^{\RP}
				)
			\to
				\bigoplus_{s \in S_{0}}
					\sheafhom_{S_{\RPS}}(
						H,
						i_{s, \ast} \Q_{p} / \Z_{p}
					).
		\]
	We have
		$
				\sheafhom_{S_{\RPS}}(
					H,
					i_{s, \ast} \Q_{p} / \Z_{p}
				)
			=
				0
		$
	since $H_{s}$ is connected by Proposition \ref{0021}.
	Thus
		\[
				j_{\ast} G_{K}
			\cong
				\sheafhom_{S_{\RPS}}(
					H,
					\nu_{\infty}(1)_{S}^{\RP}
				),
		\]
	which shows the desired mapping property.
\end{proof}

The second fundamental case is:

\begin{proposition} \label{0004}
	Let $G_{K}$ be a wound unipotent group over $K$
	with \'etale dual.
	Then $G_{K}$ is connected and admits a N\'eron model.
\end{proposition}

\begin{proof}
	Let $H_{K}$ be the dual of $G_{K}$.
	By taking a finite Galois extension of $K$
	and using \cite[Section 10.1, Proposition 4]{BLR90},
	we may assume that $H_{K} \cong \Z / p^{n} \Z$ for some $n \ge 1$.
	Then $G_{K}^{\RP} \cong \nu_{n}(1)_{K}^{\RP}$.
	Consider the exact sequence
		\[
				0
			\to
				\Gm
			\to
				\Gm^{(1 / p^{n})}
			\to
				\Gm^{(1 / p^{n})} / \Gm
			\to
				0
		\]
	of smooth group schemes over $K$.
	The first morphism, after relative perfection,
	becomes the $p^{n}$-th power map
	$\Gm^{\RP} \to \Gm^{\RP}$.
	Hence the entire sequence becomes an exact sequence
		\[
				0
			\to
				\Gm^{\RP}
			\stackrel{p^{n}}{\to}
				\Gm^{\RP}
			\to
				(\Gm^{(1 / p^{n})} / \Gm)^{\RP}
			\to
				0
		\]
	after relative perfection.
	Comparing this with \eqref{0015},
	we know that
		\begin{equation} \label{0074}
				\nu_{n}(1)_{K}^{\RP}
			\cong
				(\Gm^{(1 / p^{n})} / \Gm)^{\RP}.
		\end{equation}
	Since $\Gm^{(1 / p^{n})} / \Gm$ has a N\'eron model
	by \cite[Section 10.1, Example 11]{BLR90},
	its relative perfection has a relatively perfect N\'eron model
	by Proposition \ref{0062}.
	Thus $G_{K}$ has a N\'eron model
	by Proposition \ref{0000}.
	Since the relative perfection of a smooth $K$-scheme
	does not change the set of (open) irreducible components
	by \cite[Proposition 8.13]{BS24},
	the connectedness of $\Gm^{(1 / p^{n})} / \Gm$
	implies the connectedness of $\nu_{n}(1)_{K}^{\RP}$
	and hence of $G_{K}$.
\end{proof}

We will see later in Proposition \ref{0007} that
$G_{K}$ having an \'etale dual is equivalent to
$G_{K}$ being unirational.
But we do not have to use this equivalence
for proving Theorem \ref{0050}.

\begin{remark}
	If $r > 1$, an analogue of the exact sequence \eqref{0060}
	still exists by purity
	for logarithmic Hodge-Witt sheaves (\cite{Shi07})
		\[
				0
			\to
				\nu_{\infty}(r)_{S}^{\RP}
			\to
				j_{\ast} \nu_{\infty}(r)_{K}^{\RP}
			\to
				\bigoplus_{s \in S_{0}}
					i_{s, \ast} \nu_{\infty}(r - 1)_{s}^{\RP}
			\to
				0
		\]
	as an exact sequence of sheaves over $S_{\RPS}$.
	But $\nu_{\infty}(r - 1)_{s}^{\RP}$ is no longer \'etale,
	so the proof of Proposition \ref{0003} breaks down at this point.
\end{remark}

%%%%%%%%%%%%%%%%%%%%%%%%%%%%%%%%%%%%%%%%%%%%%%%%%%%%%%%%%%%%%%%%%%%%%%%%%%%%%

\section{Existence of N\'eron models: d\'evissage}
\label{0057}

Assume that $S \ne \Spec K$ and $r = 1$.
Now we will prove the first half
of Theorem \ref{0050} \eqref{0052},
namely the part for Conjecture \ref{0012} \eqref{0047}.
We first need to check that
\cite[Section 10.2, Lemma 12]{BLR90}
on filtering connected wound unipotent groups
preserves the connectedness of the duals:

\begin{proposition} \label{0073}
	Let $G_{K}$ be a connected wound unipotent group over $K$
	with connected dual.
	\begin{enumerate}
	\item \label{0071}
		Any closed subgroup of $G_{K}$ or $G_{K}^{\RP}$ has connected dual.
	\item \label{0072}
		$G_{K}$ admits a finite filtration by closed subgroups
		whose successive subquotients are connected wound unipotent groups
		killed by $p$ with connected dual.
	\end{enumerate}
\end{proposition}

\begin{proof}
	\eqref{0071}
	Let $0 \to G_{K}' \to G_{K}^{\RP} \to G_{K}'' \to 0$ be an exact sequence
	of relatively perfect group schemes over $K$.
	It induces an exact sequence
		\begin{align*}
			&
					\sheafhom_{K_{\RP}}(G_{K}^{\RP}, \nu_{\infty}(1))
				\to
					\sheafhom_{K_{\RP}}(G_{K}', \nu_{\infty}(1))
			\\
			&	\to
					\sheafext^{1}_{K_{\RP}}(G_{K}'', \nu_{\infty}(1))
				\to
					\sheafext^{1}_{K_{\RP}}(G_{K}^{\RP}, \nu_{\infty}(1))
		\end{align*}
	in $\Ab(K_{\RP})$,
	where $\sheafext^{1}_{K_{\RP}}$ denotes
	the first sheaf-Ext functor for $\Ab(K_{\RP})$.
	The first term is the dual of $G_{K}$
	and hence connected by assumption.
	The third term has a finite filtration
	with successive subquotients all isomorphic to $\Ga^{\RP}$
	by \cite[Proposition 3.3]{Suz24}
	and, in particular, connected.
	The fourth term is zero
	by \cite[Proposition 3.4]{Suz24} and the woundness of $G_{K}$.
	These together imply the connectedness of the second term as desired.
	
	\eqref{0072}
	Let $G_{K}[p]^{\natural} \subset G_{K}[p]$ be
	the maximal smooth closed subgroup (\cite[Definition 8.3]{BS24})
	and $(G_{K}[p]^{\natural})^{0}$ its identity component.
	It is connected wound unipotent with connected dual
	by \eqref{0071}.
	If $G_{K}$ is killed by some $p^{n}$,
	then $G_{K} / (G_{K}[p]^{\natural})^{0}$ is killed by $p^{n - 1}$.
	Therefore, by induction,
	it is enough to show that
	$G_{K} / (G_{K}[p]^{\natural})^{0}$ is connected wound with connected dual.
	The connectedness follows from that of $G_{K}$.
	The exact sequence
		\[
				0
			\to
				G_{K}[p] / (G_{K}[p]^{\natural})^{0}
			\to
				G_{K} / (G_{K}[p]^{\natural})^{0}
			\stackrel{p}{\to}
				G_{K}
		\]
	shows that
	$G_{K} / (G_{K}[p]^{\natural})^{0}$
	is wound.
	We have
	$(G_{K}[p]^{\natural})^{\RP} \isomto G_{K}[p]^{\RP}$
	by \cite[Proposition 8.4]{BS24}.
	Hence relative perfection induces an exact sequence
		\[
				0
			\to
				\pi_{0}(G_{K}[p])
			\to
				\bigl(
					G_{K} / (G_{K}[p]^{\natural})^{0}
				\bigr)^{\RP}
			\to
				G_{K}^{\RP}.
		\]
	The image of the second morphism
		$
				\bigl(
					G_{K} / (G_{K}[p]^{\natural})^{0}
				\bigr)^{\RP}
			\to
				G_{K}^{\RP}
		$
	is connected wound with connected dual
	by \eqref{0071}.
	By Proposition \ref{0004},
	$\pi_{0}(G_{K}[p])$ has connected dual.
	Since a short exact sequence of relatively perfect wound unipotent groups
	gives rise to a short exact sequence of their duals
	(\cite[Proposition 3.4]{Suz24}),
	we know that the term
		$
			\bigl(
				G_{K} / (G_{K}[p]^{\natural})^{0}
			\bigr)^{\RP}
		$
	has connected dual.
\end{proof}

Now we can use Proposition \ref{0003}:

\begin{proposition} \label{0006}
	Let $G_{K}$ be a wound unipotent group over $K$
	with connected dual.
	Then it admits a quasi-compact N\'eron model over $S$.
\end{proposition}

\begin{proof}
	By \cite[Section 7.5, Proposition 1 (b)]{BLR90},
	we may assume that $G_{K}$ is connected.
	The same reference together with Proposition \ref{0073} \eqref{0072}
	shows that we may further assume that $G_{K}$ is killed by $p$.
	This case is Proposition \ref{0003}.
\end{proof}

\begin{proposition} \label{0005}
	Let $G_{K}$ be a wound unipotent group.
	Then it admits a N\'eron model over $S$.
\end{proposition}

\begin{proof}
	This follows from
	Propositions \ref{0004} and \ref{0006}
	by Proposition \ref{0018} \eqref{0017}
	and the connected-\'etale sequence
	for the dual of $G_{K}$.
\end{proof}

\begin{theorem} \label{0009}
	Let $G_{K}$ be a connected smooth
	algebraic group over $K$ not containing $\Ga$.
	Then it has a N\'eron model over $S$.
\end{theorem}

\begin{proof}
	First note that it is enough to prove the statement
	over a finite Galois extension of $K$
	by \cite[Section 10.1, Proposition 4]{BLR90}.
	Let $T_{K}$ be the maximal torus of $G_{K}$,
	which has a N\'eron model.
	As above, we may assume that $T_{K}$ is split.
	Let $G_{K, 1} = G_{K} / T_{K}$.
	Then any morphism $\Ga \to G_{K, 1}$ defines
	an extension of $\Ga$ by $\Gm^{\dim T_{K}}$,
	which has to be a split extension.
	Hence the morphism $\Ga \to G_{K, 1}$
	factors through $G_{K}$
	and thus is zero.
	Hence $G_{K, 1}$ does not contain $\Ga$.
	Let $A_{K}$ be the maximal abelian variety
	quotient of $G_{K, 1}$
	(\cite[Section 9.2, Theorem 1]{BLR90})
	and $N_{K}$ the kernel of $G_{K, 1} \onto A_{K}$.
	Then $N_{K}$ is unipotent
	(not necessarily smooth).
        Let $N_{K}^{\natural} \subset N_{K}$ be
        the maximal smooth closed subgroup (\cite[Definition 8.3]{BS24}),
        which is wound and induces an isomorphism
        $(N_{K}^{\natural})^{\RP} \isomto N_{K}^{\RP}$
        (\cite[Proposition 8.4]{BS24}).
        Hence $N_{K}^{\RP}$ is wound unipotent
	and hence has a relatively perfect N\'eron model
	by Proposition \ref{0005}.
	We have an exact sequence
	$0 \to N_{K}^{\RP} \to G_{K, 1}^{\RP} \to A_{K}^{\RP}$.
	Since $A_{K}^{\RP}$ has a quasi-compact
	relatively perfect N\'eron model,
	so does the image of the morphism
	$G_{K, 1}^{\RP} \to A_{K}^{\RP}$
	by Proposition \ref{0018} \eqref{0016}.
	Therefore $G_{K, 1}^{\RP}$ has
	a relatively perfect N\'eron model
	by Proposition \ref{0018} \eqref{0017}.
	Hence $G_{K}$ has a N\'eron model by the same reason.
\end{proof}

%%%%%%%%%%%%%%%%%%%%%%%%%%%%%%%%%%%%%%%%%%%%%%%%%%%%%%%%%%%%%%%%%%%%%%%%%%%%%

\section{Quasi-compactness of N\'eron models}
\label{0058}

Assume that $S \ne \Spec K$ and $r = 1$.
We will prove the second half
of Theorem \ref{0050} \eqref{0052},
namely the part for Conjecture \ref{0012} \eqref{0048}.
We need to describe wound unipotent groups
with no non-trivial unirational subgroup
in terms of duals:

\begin{proposition} \label{0019}
	Let $G_{K}$ be a wound unipotent group over $K$.
	Then it has no non-trivial unirational subgroup
	if and only if its dual is connected.
\end{proposition}

\begin{proof}
	Let $H_{K}$ be the dual of $G_{K}$.
	The stated conditions are stable
	under finite separable extensions of $K$
	by \cite[Section 10.3, Remark 4]{BLR90}.
	Hence we may replace $K$ by $K^{\sep}$.
	A wound unipotent group over $K = K^{\sep}$
	with $[K : K^{p}] = p$ is unirational
	if and only if it is a quotient
	of a finite product of groups of the form
	$\Gm^{(1 / p^{n})} / \Gm$ with various $n \ge 1$
	by \cite[Proposition 2.5]{Ach19b}.
	Therefore $G_{K}$ has no non-trivial unirational subgroup
	if and only if
	$G_{K}$ has no non-zero morphism from $\Gm^{(1 / p^{n})} / \Gm$
	if and only if
	$G_{K}^{\RP}$ has no non-zero morphism from $\nu(1)_{K}^{\RP}$
	(see \eqref{0074})
	if and only if
	$H_{K}$ has no non-zero morphism to $\Z / p \Z$
	if and only if
	$H_{K}$ is connected.
\end{proof}

\begin{theorem} \label{0013}
	Let $G_{K}$ be a connected smooth algebraic group over $K$
	with no non-zero unirational subgroup.
	Then it has a quasi-compact N\'eron model over $S$.
\end{theorem}

\begin{proof}
	Keep the notation of the proof of Theorem \ref{0009}.
	Then under the extra assumptions of Theorem \ref{0013},
	we have $T_{K} = 0$,
	and the dual of $N_{K}^{\RP}$ is connected
	by Proposition \ref{0019},
	so that now we can apply Proposition \ref{0006} instead of \ref{0005}.
\end{proof}

%%%%%%%%%%%%%%%%%%%%%%%%%%%%%%%%%%%%%%%%%%%%%%%%%%%%%%%%%%%%%%%%%%%%%%%%%%%%%

\section{Non-existence of N\'eron models}
\label{0040}

Now the number $r$ (in $[K : K^{p}] = p^{r}$)
is arbitrary.
In this section,
we will prove Theorem \ref{0050} \eqref{0053}.

Set $K' = K^{1 / p}$.
Let $S'$ be the normalization of $S$ in $K'$.
Let $\Omega_{S}^{r} = \Omega_{S / \F_{p}}^{r}$ be
the sheaf of absolute differentials
(which is a line bundle over $S$).
View it as a vector group over $S$
(of relative dimension $1$).
Let
	\[
			C - W^{\ast}
		\colon
			\Res_{S' / S} \Omega_{S'}^{r}
		\onto
			\Omega_{S}^{r}
	\]
be the $\Order_{S}$-linear Cartier operator $C$
minus the formal $p$-th power $W^{\ast}$
(see \cite[Section 2]{AM76}
in the perfect residue field case
or the explicit construction below).
It is a smooth surjective morphism
of smooth group schemes of finite type.
Let $\nu(r)_{S}$ be the kernel of $C - W^{\ast}$.
All these constructions equally apply to $K$
in place of $S$.
In particular,
we have a smooth unipotent group $\nu(r)_{K}$ over $K$,
which is the generic fiber of $\nu(r)_{S}$.

\begin{theorem} \label{0032}
	Assume that $r > 1$ and $S$ has infinitely many points.
	Then $\nu(r)_{K}$ does not admit a N\'eron model over $S$.
\end{theorem}

We will prove this below.
For the moment, we only assume that $S \ne \Spec K$
and do not make the assumptions of Theorem \ref{0032}.

We explicitly describe
$\nu(r)_{K}$ and $\nu(r)_{S}$.
Let $s \in S$ be a closed point.
Let $t_{1}, \dots, t_{r - 1} \in \Order_{S, s}$
be a lift of a $p$-basis of $k(s)$.
Let $\pi \in \Order_{S, s}$ be a prime element.
Then $\{t_{1}, \dots, t_{r - 1}, \pi\}$ form
a $p$-basis of $\Order_{S, s}$
and hence of $K$.
Therefore
	\begin{equation} \label{0027}
				\Omega_{K}^{r}
			\cong
					\Ga
					\dlog t_{1}
				\wedge
					\dots
				\wedge
					\dlog t_{r - 1}
				\wedge
					\dlog \pi,
	\end{equation}
	\begin{equation} \label{0028}
			\Res_{K' / K} \Omega_{K'}^{r}
		\cong
			\bigoplus_{0 \le i(1), \dots, i(r) \le p - 1}
					\Ga
					t_{1}^{i(1) / p}
					\cdots t_{r - 1}^{i(r - 1) / p}
					\pi^{i(r) / p}
					\dlog t_{1}^{1 / p}
				\wedge
					\dots
				\wedge
					\dlog t_{r - 1}^{1 / p}
				\wedge
					\dlog \pi^{1 / p}.
	\end{equation}
The morphism $C$ is linear
and the morphism $W^{\ast}$ is Frobenius-linear.
For each $0 \le i(1), \dots, i(r - 1) \le p - 1$, we have
	\begin{equation} \label{0029}
		\begin{aligned}
			&
					C \Bigl(
							t_{1}^{i(1) / p}
							\cdots t_{r - 1}^{i(r - 1) / p}
							\pi^{i(r) / p}
							\dlog t_{1}^{1 / p}
						\wedge
							\dots
						\wedge
							\dlog t_{r - 1}^{1 / p}
						\wedge
							 \dlog \pi^{1 / p}
					\Bigr)
			\\
			&	=
					\begin{cases}
								\dlog t_{1}
							\wedge
								\dots
							\wedge
								\dlog t_{r - 1}
							\wedge
								\dlog \pi
						&
							\text{if } i(1) = \dots = i(r) = 0,
						\\
							0
						&
							\text{else},
					\end{cases}
		\end{aligned}
	\end{equation}
	\begin{equation} \label{0030}
		\begin{aligned}
			&
					W^{\ast} \Bigl(
							t_{1}^{i(1) / p}
							\cdots t_{r - 1}^{i(r - 1) / p}
							\pi^{i(r) / p}
							\dlog t_{1}^{1 / p}
						\wedge
							\dots
						\wedge
							\dlog t_{r - 1}^{1 / p}
						\wedge
							\dlog \pi^{1 / p}
					\Bigr)
			\\
			&	=
						t_{1}^{i(1)}
						\cdots t_{r - 1}^{i(r - 1)}
						\pi^{i(r)}
						\dlog t_{1}
					\wedge
						\dots
					\wedge
						\dlog t_{r - 1}
					\wedge
						\dlog \pi.
		\end{aligned}
	\end{equation}
With these coordinate presentations,
the morphism $C - W^{\ast}$ over $K$ is given by
the morphism $\Ga^{p^{r}} \to \Ga$ defined by
	\begin{equation} \label{0031}
		\begin{aligned}
			&
					(x_{i(1) \cdots i(r)})
					_{0 \le i(1), \dots, i(r) \le p - 1}
			\\
			&	\mapsto
						x_{0 \cdots 0}
					-
						\sum_{0 \le i(1), \dots, i(r) \le p - 1}
							x_{i(1) \cdots i(r)}^{p}
							t_{1}^{i(1)}
							\cdots t_{r - 1}^{i(r - 1)}
							\pi^{i(r)}.
		\end{aligned}
	\end{equation}
Its kernel is $\nu(r)_{K}$.
From this presentation,
we can see that if $r > 0$,
then $\nu(r)_{K}$ is non-zero, connected
(by changing the base to the algebraic closure of $K$)
and wound (by \cite[Example B.2.4]{CGP15}).

Similarly, over $\Order_{S, s}$, we have
	\begin{gather*}
				\Omega_{S}^{r}
			\cong
					\Ga \dlog t_{1}
				\wedge
					\dots
				\wedge
					\dlog t_{r - 1}
				\wedge
					d \pi,
		\\
				\Res_{S' / S} \Omega_{S'}^{r}
			\cong
				\bigoplus_{0 \le i(1), \dots, i(r) \le p - 1}
						\Ga
						t_{1}^{i(1) / p}
						\cdots t_{r - 1}^{i(r - 1) / p}
						\pi^{i(r) / p}
						\dlog t_{1}^{1 / p}
					\wedge
						\dots
					\wedge
						\dlog t_{r - 1}^{1 / p}
					\wedge
						d \pi^{1 / p}.
	\end{gather*}
Note that the final terms are $d \pi$ and $d \pi^{1 / p}$,
not $\dlog \pi$ and $\dlog \pi^{1 / p}$.
The morphism $C - W^{\ast}$ over $\Order_{S, s}$ is given by
the morphism $\Ga^{p^{r}} \to \Ga$ defined by
	\begin{equation} \label{0033}
		\begin{aligned}
			&
					(y_{i(1) \cdots i(r)})
					_{0 \le i(1), \dots, i(r) \le p - 1}
			\\
			&	\mapsto
						y_{0 \cdots 0 \, p - 1}
					-
						\sum_{0 \le i(1), \dots, i(r) \le p - 1}
							y_{i(1) \cdots i(r)}^{p}
							t_{1}^{i(1)}
							\cdots t_{r - 1}^{i(r - 1)}
							\pi^{i(r)}.
		\end{aligned}
	\end{equation}
Its kernel is $\nu(r)_{S}$.

The change of coordinates
$(y_{i(1) \cdots i(r)}) \mapsto (x_{i(1) \cdots i(r)})$
for the source $\Ga^{p^r}$ is given by
	\begin{equation} \label{0076}
			x_{i(1) \cdots i(r)}
		=
			\begin{cases}
					y_{i(1) \cdots i(r - 1) \, p - 1} \pi
				&
					\text{if } i(r) = 0,
				\\
					y_{i(1) \cdots i(r - 1) \, i(r) - 1}
				&
					\text{else}.
			\end{cases}
	\end{equation}
The change of coordinates $y \mapsto x$
for the target $\Ga$ is given by $x = y \pi$.

Next we explicitly construct
local N\'eron models of $\nu(r)_{K}$.
Let $\Omega_{S}^{r}(\log s)$ be
the sheaf of absolute differentials
with log pole at $s$.
Let $s' \in S'$ be the point over $s$.
Let $\Omega_{S'}^{r}(\log s')$ be
the sheaf of absolute differentials
with log pole at $s'$.
Let
	\[
			C - W^{\ast}
		\colon
			\Res_{S' / S} \Omega_{S'}^{r}(\log s')
		\onto
			\Omega_{S}^{r}(\log s)
	\]
be the $\Order_{S}$-linear Cartier operator $C$
minus the formal $p$-th power $W^{\ast}$.
Define $\nu(r)_{\Order_{S, s}}^{\sim}$ to be
the kernel of this $C - W^{\ast}$.
These are described by the same formulas
\eqref{0027}--\eqref{0031} as for $\nu(r)_{K}$.
Hence $\nu(r)_{\Order_{S, s}}^{\sim}$ is
the kernel of the morphism \eqref{0031},
this time over $\Order_{S, s}$.

\begin{proposition} \label{0035}
	Let $s \in S$ be a closed point.
	Then the local N\'eron model of $\nu(r)_{K}$ at $s$ is
	given by $\nu(r)_{\Order_{S, s}}^{\sim}$.
\end{proposition}

\begin{proof}
	We need to show that any tuple $(x_{i(1) \cdots i(r)})$
	of elements of $K_{s}^{sh}$
	that maps to zero by \eqref{0031}
	is actually a tuple of elements of $\Order_{S, s}^{sh}$.
	Let $v_{s}$ be the normalized valuation at $s$.
	
	We show that for any $0 \le i(r) \le p - 1$, we have an equality
		\begin{align*}
			&
					v_{s} \left(
						\sum_{0 \le i(1), \dots, i(r - 1) \le p - 1}
							x_{i(1) \cdots i(r)}^{p}
							t_{1}^{i(1)} \cdots t_{r - 1}^{i(r - 1)}
					\right)
			\\
			&	=
					p \min_{0 \le i(1), \dots, i(r - 1) \le p - 1}
						v_{s}(x_{i(1) \cdots i(r)}).
		\end{align*}
	The inequality $\ge$ is obvious.
	Dividing the $x_{i(1) \cdots i(r)}$ by an integer power of $\pi$,
	we may assume that the right-hand side is zero.
	Consider the element inside $v_{s}(\;\cdot\;)$ of the left-hand side
	and its reduction modulo $\pi$:
		\[
				\sum_{0 \le i(1), \dots, i(r - 1) \le p -1}
					\Bar{x}_{i(1) \cdots i(r)}^{p}
					\bar{t}_{1}^{i(1)} \cdots \bar{t}_{r - 1}^{i(r - 1)}
			\in
				k(s),
		\]
	where the bars denote reduction.
	This is non-zero
	since $\Bar{t}_{1}, \dots, \Bar{t}_{r - 1}$ are a $p$-basis of $k(s)$
	and some of the $\Bar{x}_{i(1) \cdots i(r)}$ are non-zero by assumption.
	Thus the left-hand side is also zero.
	
	From this equality, it follows
        by elementary calculations of
        discrete valuations
        that none of $(x_{i(1) \cdots i(r)})$
	can have negative valuation.
	Hence all of them are elements of $\Order_{S, s}^{sh}$.
\end{proof}

Now, over $\Order_{S, s}$, we have two models
$\nu(r)_{S} \times_{S} \Order_{S, s}$ and
$\nu(r)_{\Order_{S, s}}^{\sim}$ of $\nu(r)_{K}$,
which have different fibers:

\begin{proposition} \label{0034}
	Assume that $r > 1$.
	Let $s \in S$ be a closed point.
	Then the fiber of $\nu(r)_{\Order_{S, s}}^{\sim}$ at $s$
	and the fiber of $\nu(r)_{S}$ at $s$
	are both connected and not isomorphic.
\end{proposition}

\begin{proof}
	These fibers are
	$\nu(r - 1)_{s} \times \Ga^{p^{r - 1} (p - 1)}$
	and $\Ga^{p^{r} - 1}$, respectively,
	by \eqref{0031} and \eqref{0033}.
	The assumption $r > 1$ implies that
	$\nu(r - 1)_{s}$ is non-zero, connected and wound
	as we saw after \eqref{0031}.
\end{proof}

We are ready to prove the theorem.

\begin{proof}[Proof of Theorem \ref{0032}]
	Assume that $r > 1$ and $S$ has infinitely many points.
	Suppose for contradiction that
	$\nu(r)_{K}$ admits a N\'eron model over $S$, say $G$.
	By Propositions \ref{0035} and \ref{0034},
	$G$ has connected fibers and hence is quasi-compact.
	We have a natural morphism $\nu(r)_{S} \to G$ over $S$,
	which is an isomorphism over $K$.
	Since $\nu(r)_{S}$ and $G$ are both of finite type over $S$,
	it follows by spreading out that
	there exists a dense open subscheme $U$ of $S$
	such that $\nu(r)_{S} \to G$ is an isomorphism over $U$.
	Hence for any closed point $s$ of $U$,
	the localized morphism
		\[
				\nu(r)_{S} \times_{S} \Order_{S, s}
			\to
				G \times_{S} \Order_{S, s}
			=
				\nu(r)_{\Order_{S, s}}^{\sim}
		\]
	over $\Order_{S, s}$ is an isomorphism.
	But as $S$ is assumed to have
	infinitely many closed points,
	$U$ cannot be $\Spec K$.
	Hence this contradicts Proposition \ref{0034}.
\end{proof}

The only remaining statement to prove
in Theorem \ref{0050} \eqref{0053} is:

\begin{proposition} \label{0061}
	The group $\nu(r)_{K}$ is unirational.
\end{proposition}

\begin{proof}
	Let $t_{1}, \dots, t_{r}$ be
	a $p$-basis of $K$.
	Let $\Proj^{r - 1}(\F_{p})$ be
	the set of $\F_{p}$-points of the projective $(r - 1)$-space.
	For an element
	$\underline{i} \in \Proj^{r - 1}(\F_{p})$,
	let $I_{\underline{i}}$ be
	the fiber of the natural quotient map
	$\F_{p}^{r} \setminus \{0\} \onto \Proj^{r - 1}(\F_{p})$
	over $\underline{i}$.
	Define a $K$-subgroup
	$G_{\underline{i}} \subset \nu(r)_{K}$
	by the additional equations
	$x_{i(1) \cdots i(r)} = 0$ in $\Ga^{p^{r}}$
	over
		$
				(i(1), \dots, i(r))
			\in
				\F_{p}^{r} \setminus I_{\underline{i}} \cup \{0\}
		$,
	where we identified the set $\{0, 1, \dots, p - 1\}$ with $\F_{p}$.
	Hence $G_{\underline{i}}$ is the $(p - 1)$-dimensional group
	defined by the single equation
		\[
					x_{0 \cdots 0}
				-
					x_{0 \cdots 0}^{p}
				-
					\sum_{(i(1), \dots, i(r)) \in I_{\underline{i}}}
						x_{i(1) \cdots i(r)}^{p}
						t_{1}^{i(1)} \cdots t_{r}^{i(r)}
			=
				0
		\]
	in the affine $p$-space.
	It is isomorphic to the subgroup of $\Ga^{p}$
	defined by the equation
		\[
					x_{0}
				-
					\sum_{i = 0}^{p - 1}
						x_{i}^{p} t^{i}
			=
				0,
			\quad \text{or} \quad
					y_{p - 1}
				-
					\sum_{i = 0}^{p - 1}
						y_{i}^{p} t^{i}
			=
				0,
		\]
	where $t = t_{1}^{i(1)} \cdots t_{r}^{i(r)}$
	with any choice of
	$(i(1), \dots, i(r)) \in I_{\underline{i}}$
	and $(x_{i})$ and $(y_{i})$ are related by a change of coordinates
	of the type \eqref{0076}.
	Therefore $G_{\underline{i}}$ is isomorphic to
	$(\Res_{K(t^{1 / p}) / K} \Gm) / \Gm$
	by \cite[Chapter VI, Proposition 5.3]{Oes84}
	and hence unirational.
	The Lie algebra of $G_{\underline{i}}$
	as a subspace of $\Ga^{p^{r}}$
	is defined by the equations
	$x_{i(1) \cdots i(r)} = 0$
	over
		$
				(i(1), \dots, i(r))
			\in
				\F_{p}^{r} \setminus I_{\underline{i}}
		$.
	Now consider the morphism
		\[
				\prod_{\underline{i} \in \Proj^{r - 1}(\F_{p})}
					G_{\underline{i}}
			\to
				\nu(r)_{K}
		\]
	defined by the sum of the inclusion maps
	$G_{\underline{i}} \into \nu(r)_{K}$
	over all $\underline{i} \in \Proj^{r - 1}(\F_{p})$.
	Since the Lie algebra of $\nu(r)_{K}$ is
	defined by the equation
	$x_{0 \cdots 0} = 0$,
	this morphism induces an isomorphism on the Lie algebras.
	Hence it is finite \'etale and, in particular, surjective.
	The unirationality of $\nu(r)_{K}$ then follows.
\end{proof}

%%%%%%%%%%%%%%%%%%%%%%%%%%%%%%%%%%%%%%%%%%%%%%%%%%%%%%%%%%%%%%%%%%%%%%%%%%%%%

\section{Groups invisible by relative perfection}
\label{0059}

We will prove Theorem \ref{0051} in the next section.
As preliminaries, in this section,
we study the kernel of the map $G^{(1 / p)} \onto G$
given in \eqref{0036}
for smooth connected unipotent groups $G$.

Assume $S = \Spec K$.
Set $K' = K^{1 / p}$.
Define $N_{0}$ to be the kernel of
the morphism $\Ga^{(1 / p)} \onto \Ga$ given in \eqref{0036}.
Hence we have an exact sequence
	\[
			0
		\to
			N_{0}
		\to
			\Ga^{(1 / p)}
		\to
			\Ga
		\to
			0.
	\]
The structure of $N_{0}$ is, in a sense, simple:

\begin{proposition} \label{0039}
	The endomorphism ring $\End_{K}(N_{0})$ is
	canonically isomorphic to $K'$.
\end{proposition}

This does not mean that
the only closed subgroups of $N_{0}$ are zero or $N_{0}$.

\begin{proof}
	First note that $\Ga^{(1 / p)}$
	(the Weil restriction of $\Ga$
	along the absolute Frobenius $\Spec K \to \Spec K$)
	can be identified with the Weil restriction of $\Ga$ over $K'$
	along the natural morphism $\Spec K' \to \Spec K$.
	The latter Weil restriction as a functor sends
	a $K$-algebra $A$ to $A \tensor_{K} K'$.
	This final functor is given by $\Ga \tensor_{K} K'$
	with the natural $K$-linear action on $\Ga$.
	Thus $\Ga^{(1 / p)} \cong \Ga \tensor_{K} K'$.
	Under this isomorphism,
	the morphism $\Ga^{(1 / p)} \onto \Ga$ is given by
	$x \tensor a' \mapsto x^{p} a'^{p}$
	(with $x \in \Ga$ and $a' \in K'$).
	
	We define a ring homomorphism
	$K' \to \End_{K}(N_{0})$.
	Let $a' \in K'$.
	Then we have a morphism
	$\Ga^{(1 / p)} \to \Ga^{(1 / p)}$ over $K$
	given by multiplication by $a'$.
	We also have a morphism $\Ga \to \Ga$
	given by multiplication by $a'^{p} \in K$.
	They are compatible with $\Ga^{(1 / p)} \onto \Ga$
	and so define a morphism $N_{0} \to N_{0}$.
	The multiplication in $K'$ corresponds to
	the composition in $\End_{K}(N_{0})$.
	Thus we have a ring homomorphism
	$K' \to \End_{K}(N_{0})$.
        It is injective since $K'$ is a field.
        
	Let $\varphi \in \End_{K}(N_{0})$.
	We want to show that it belongs to the subring $K'$.
        If there exists a $K$-linear morphism
        $\psi \colon \Ga^{(1 / p)} \to \Ga^{(1 / p)}$
        that makes the diagram
		\begin{equation} \label{0037}
			\begin{CD}
					0
				@>>>
					N_{0}
				@>>>
					\Ga^{(1 / p)}
				@>>>
					\Ga
				@>>>
					0
				\\ @. @V \varphi VV @V \psi VV \\
					0
				@>>>
					N_{0}
				@>>>
					\Ga^{(1 / p)}
				@>>>
					\Ga
				@>>>
					0
			\end{CD}
		\end{equation}
        commutative,
        then the induced morphism
        $\gamma \colon \Ga \to \Ga$
        that makes the diagram
		\[
			\begin{CD}
					0
				@>>>
					N_{0}
				@>>>
					\Ga^{(1 / p)}
				@>>>
					\Ga
				@>>>
					0
				\\ @. @V \varphi VV @V \psi VV @V \gamma VV \\
					0
				@>>>
					N_{0}
				@>>>
					\Ga^{(1 / p)}
				@>>>
					\Ga
				@>>>
					0
			\end{CD}
		\]
        commutative has to be $K$-linear by a degree reason,
	which implies that
	$\psi$ is the multiplication by an element of $K'$.
	Therefore it is enough to show
        the existence of such $\psi$.
	
	Let $t_{1}, \dots, t_{r}$ be a $p$-basis of $K$.
	The scheme $N_{0}$ is affine
	with affine ring $\Order(N_{0})$ given by
		\[
				\frac{
					K[
						x_{i(1) \cdots i(r)}
					\,|\,
						0 \le i(1), \dots, i(r) \le p - 1
					]
				}{
					(\sum_{i(1), \dots, i(r)}
						x_{i(1) \cdots i(r)}^{p}
						t_{1}^{i(1)} \cdots t_{r}^{i(r)}
					)
				}
		\]
	The morphism $\varphi \colon N_{0} \to N_{0}$ corresponds to
	a $K$-algebra homomorphism $\Order(N_{0}) \to \Order(N_{0})$,
	which can be written as
		\begin{equation} \label{0038}
				x_{j(1) \cdots j(r)}
			\mapsto
				\sum_{i = 0}^{p - 1}
					f_{j(1) \cdots j(r) i}
					x_{0 \cdots 0}^{i},
		\end{equation}
	where each $f_{j(1) \cdots j(r) i}$ is
	a polynomial in $x_{i(1) \cdots i(r)}$
	over $(i(1), \dots, i(r)) \ne (0, \dots, 0)$
	with coefficients in $K$.
	The additivity of $\varphi \colon N_{0} \to N_{0}$
	translates to the conditions that
	the polynomial  $f_{j(1) \cdots j(r) 0}$ gives
	a morphism $\Ga^{p^{r} - 1} \to \Ga^{p^{r} - 1}$,
	the polynomial $f_{j(1) \cdots j(r) 1}$ is constant
	and the polynomial $f_{j(1) \cdots j(r) i}$
	for $i \ge 2$ is zero.
        The assignment \eqref{0038} needs
        to map the polynomial
		$
			\sum_{i(1), \dots, i(r)}
				x_{i(1) \cdots i(r)}^{p}
				t_{1}^{i(1)} \cdots t_{r}^{i(r)}
		$
        to the zero element in $\Order(N_{0})$.
	This means an equality
		\begin{align*}
			&		\sum_{j(1), \dots, j(r)}
						f_{j(1) \cdots j(r) 0}^{p}
						t_{1}^{j(1)} \cdots t_{r}^{j(r)}
			\\
			&	=
					\sum_{\substack{
						i(1), j(1), \dots, i(r), j(r) \\
						(i(1), \dots, i(r)) \ne (0, \dots, 0)
					}}
						f_{j(1) \cdots j(r) 1}^{p}
						x_{i(1) \cdots i(r)}^{p}
						t_{1}^{i(1) + j(1)} \cdots t_{r}^{i(r) + j(r)}.
		\end{align*}
	Comparing both sides
        and using that $t_{1}, \dots, t_{r}$ form a $p$-basis of $K$,
	we know that the polynomial $f_{j(1) \cdots j(r) 0}$ is
	linear with no constant term.
	This means that \eqref{0038} is linear.
        Now define a $K$-linear morphism
        $\psi \colon \Ga^{(1 / p)} \to \Ga^{(1 / p)}$
        by the assignment \eqref{0038}.
        This $\psi$ makes the diagram \eqref{0037} commutative,
        as desired.
\end{proof}

With this, we can describe
the structure of the kernel of $G^{(1 / p)} \onto G$:

\begin{proposition} \label{0042}
	Let $G$ be a smooth connected unipotent group over $K$.
	Let $N$ be the kernel of the morphism
	$G^{(1 / p)} \onto G$.
	Then $N$ admits a finite filtration
	whose successive subquotients are all isomorphic to $N_{0}$.
	If $p G = 0$, then $N$ is isomorphic to $N_{0}^{\dim G}$.
\end{proposition}

\begin{proof}
	First assume that $p G = 0$.
	Then there exists an exact sequence
	$0 \to G \to \Ga^{n} \to \Ga \to 0$ with $n = \dim G + 1$
	by \cite[Section 10.2, Proposition 10]{BLR90}.
	Since $G$ is smooth, this sequence remains exact
	after applying the Weil restriction functor
	$(1 / p) = \Res_{K' / K}$.
	Therefore we have an exact sequence
		\[
				0
			\to
				N
			\to
				N_{0}^{n}
			\to
				N_{0}
			\to
				0.
		\]
	Since $\End_{K}(N_{0}) \cong K'$ is a field
	by Proposition \ref{0039},
	the last morphism $N_{0}^{n} \onto N_{0}$ admits a splitting.
	Hence $N \cong N_{0}^{n - 1}$.
	The case of general $G$ follows from this
	by \cite[Section 10.2, Lemma 12]{BLR90}.
\end{proof}

%%%%%%%%%%%%%%%%%%%%%%%%%%%%%%%%%%%%%%%%%%%%%%%%%%%%%%%%%%%%%%%%%%%%%%%%%%%%%

\section{Unirational wound unipotent groups}

Assume $S = \Spec K$ and $r = 1$.
For $i \ge 0$,
let $\Ext_{K}^{i}$ be the $i$-th Ext functors for the fppf site of $K$.
Note that when $G$ and $H$ are smooth group schemes over $K$,
the group $\Ext_{K}^{i}(G, H)$ does not change
if we instead use the big \'etale site of $K$.
The groups $N_{0}$ and $N$ in the previous section
do not contribute to derived $\Gm$-duals in low degrees:

\begin{proposition} \label{0043}
	We have
	$\Hom_{K}(N_{0}, \Gm) = \Ext^{1}_{K}(N_{0}, \Gm) = 0$.
\end{proposition}

\begin{proof}
	We have $\Hom_{K}(\Ga, \Gm) = \Ext^{1}_{K}(\Ga, \Gm) = 0$.
	Hence $\Hom_{K}(N_{0}, \Gm) = 0$ and we have an exact sequence
		\begin{equation} \label{0041}
				0
			\to
				\Ext^{1}_{K}(N_{0}, \Gm)
			\to
				\Ext^{2}_{K}(\Ga, \Gm)
			\to
				\Ext^{2}_{K}(\Ga^{(1 / p)}, \Gm).
		\end{equation}
	We have an exact sequence
		\[
				0
			\to
				\Gm
			\stackrel{\incl}{\to}
				\Res_{K' / K}
				\Gm
			\stackrel{\dlog}{\to}
				\Res_{K' / K}
				\Omega_{K'}^{1}
			\stackrel{C - W^{\ast}}{\to}
				\Omega_{K}^{1}
			\to
				0
		\]
	of smooth group schemes over $K$
	as in Section \ref{0040}
	(namely by \cite[Lemma (2.1)]{AM76}).
	This sequence defines a group homomorphism
		\[
				\Omega^{1}_{K}
			\to
				\Ext^{2}_{K}(\Ga, \Gm),
		\]
	which is an isomorphism by \cite[Proposition 2.7.6]{Ros23}
	(which requires $r = 1$).
	This isomorphism is $K$-linear with the natural $K$-actions
	on $\Omega^{1}_{K}$ and $\Ga$.
	Let $C_{0} \colon \Omega^{1}_{K} \to \Omega^{1}_{K}$ be
	the inverse-Frobenius-linear Cartier operator.
	Then the endomorphism
	$C_{0}$ on $\Omega^{1}_{K}$ corresponds to
	the Frobenius endomorphism $F$ on $\Ga$.
	From these, we can translate the final morphism
	in \eqref{0041} as the morphism
		\begin{gather*}
					\Omega^{1}_{K}
				\to
					\Hom_{\text{$K$-linear}}(K', \Omega^{1}_{K});
			\\
					\omega
				\mapsto
					(a' \mapsto C_{0}(a'^{p} \omega)).
		\end{gather*}
	This map is injective by an explicit calculation,
	hence $\Ext^{1}_{K}(N_{0}, \Gm) = 0$.
\end{proof}

\begin{proposition} \label{0044}
	Let $G$ be a smooth connected unipotent group over $K$.
	Let $N$ be the kernel of the morphism $G^{(1 / p)} \onto G$.
	Then $\Hom_{K}(N, \Gm) = \Ext^{1}_{K}(N, \Gm) = 0$.
\end{proposition}

\begin{proof}
	This follows from Propositions \ref{0042} and \ref{0043}.
\end{proof}

From these, we can show that passing to the relative perfection
does not lose information about extension groups by $\Gm$:

\begin{proposition} \label{0045}
	Let $G$ be a smooth connected unipotent group over $K$.
	Then
		$
				\Ext^{1}_{K}(G, \Gm)
			\isomto
				\Ext^{1}_{K_{\RP}}(G^{\RP}, \Gm^{\RP})
		$.
\end{proposition}

\begin{proof}
	We have
		\[
				\Ext^{1}_{K_{\RP}}(G^{\RP}, \Gm^{\RP})
			\cong
				\Ext^{1}_{K}(G^{\RP}, \Gm)
			\cong
				\dirlim_{n}
					\Ext^{1}_{K}(G^{(1 / p^{n})}, \Gm).
		\]
	Hence the result follows from Proposition \ref{0044}.
\end{proof}

Now we can apply Achet's result
on finiteness of $\Ext^{1}_{K}(G, \Gm)$
for unirational $G$ to prove that
unirational groups are dual to \'etale groups:

\begin{proposition} \label{0007}
	Let $G_{K}$ be a wound unipotent group over $K$ with dual $H_{K}$.
	Then $G_{K}$ is unirational if and only if $H_{K}$ is \'etale.
	The group $\pi_{0}(H_{K})$ is dual to
        the maximal unirational subgroup of $G_{K}$.
	An extension of a wound unipotent unirational group
	by a wound unipotent unirational group is unirational.
\end{proposition}

\begin{proof}
	By \cite[Section 10.3, Remark 4]{BLR90},
	$G_{K}$ is unirational if and only if
	$G_{K} \times_{K} K^{\sep}$ is unirational.
	Hence we may assume that $K = K^{\sep}$.
	Suppose $G_{K}$ is unirational.
	Then $\Ext^{1}_{K}(G_{K}, \Gm)$ is finite
	by \cite[Theorem 2.8]{Ach19b}.
	Note that $G_{K}$, being unipotent, is killed by a power of $p$.
	By \eqref{0015} and Proposition \ref{0045}, we have
		\[
				\Ext^{1}_{K}(G_{K}, \Gm)
			\cong
				\Hom_{K_{\RP}}(
					G_{K}^{\RP},
					\nu_{\infty}(1)_{K}^{\RP}
				)
			\cong
				H_{K}(K).
		\]
	Hence $H_{K}(K)$ is finite.
	Since $H_{K}(K)$ is dense in $H_{K}$
	by \cite[Proposition 8.15]{BS24},
	it follows that $H_{K}$ is \'etale.
	
	Suppose next that $H_{K}$ is \'etale.
	We may assume that $H_{K} = \Z / p^{n} \Z$.
	Then $G_{K}^{\RP} = \nu_{n}(1)_{K}^{\RP}$
	by \cite[Theorem 3.2 (i)]{Kat86}.
	Hence $G_{K}$ is covered by $\Gm^{(1 / p^{m})}$
        for some $m$.
        Since $\Gm^{(1 / p^{m})}$ is rational,
        this implies that
        $G_{K}$ is unirational.
	The rest of the statements then follow by duality.
\end{proof}

\begin{theorem} \label{0049}
	The category of relative perfections of
	unirational wound unipotent groups over $K$
	is canonically equivalent to
	the category of $p$-primary
	finite \'etale group schemes over $K$.
	Under this correspondence,
	$\nu_{n}(1)_{K}^{\RP}$ corresponds to $\Z / p^{n} \Z$.
\end{theorem}

\begin{proof}
	The two categories in question are contravariantly equivalent
	via the duality functor
	$\sheafhom_{K_{\RP}}(\;\cdot\;, \nu_{\infty}(1)_{K}^{\RP})$
	by Proposition \ref{0007}.
	But the latter category is also contravariantly equivalent to itself
	via the Pontryagin duality functor
	$\sheafhom_{K_{\et}}(\;\cdot\;, \Q_{p} / \Z_{p})$.
\end{proof}

This proves Theorem \ref{0051}.
We give a simple application,
which is a converse to Conjecture \ref{0012} \eqref{0048}
and a variant of
\cite[Section 10.3, Theorem 5 (b)]{BLR90}.
We return to general $S$ (not necessarily $\Spec K$).

\begin{proposition} \label{0064}
	Assume that $S$ has infinitely many points
        and $r = 1$.
	Let $G_{K}$ be a connected smooth group scheme over $K$.
	Assume that $G_{K}$ admits a quasi-compact N\'eron model.
	Then it has no non-zero unirational subgroup.
\end{proposition}

\begin{proof}
	Having a N\'eron model, $G_{K}$ does not contain $\Ga$.
	A non-trivial algebraic torus over $K$
	never admits a quasi-compact N\'eron model
	since $S$ has infinitely many closed points.
	Hence by \cite[Section 7.1, Corollary 6]{BLR90},
	the maximal torus of $G_{K}$ is zero.
	Therefore any unirational subgroup of $G_{K}$ has to be unipotent.
	Let $M_{K}$ be the dual
	of the maximal unirational subgroup of $G_{K}$.
	It is a finite \'etale group scheme over $K$
	by Theorem \ref{0049}.
	By shrinking $S$ if necessary,
	we may assume that the Galois action of $K$
	on the geometric points of $M_{K}$ is unramified everywhere.
	Since N\'eron models do not change by finite unramified extensions,
	we may assume that $M_{K}$ is constant.
	If $M_{K}$ is non-zero,
	then it contains a copy of $\Z / p \Z$,
	whose dual is $\nu(1)_{K}^{\RP}$.
	But the N\'eron model of
	$\nu(1)_{K} \cong \Gm^{(1 / p)} / \Gm$
	is not quasi-compact
	by \cite[Section 10.1, Example 11]{BLR90}.
	Therefore $M_{K} = 0$.
\end{proof}

%%%%%%%%%%%%%%%%%%%%%%%%%%%%%%%%%%%%%%%%%%%%%%%%%%%%%%%%%%%%%%%%%%%%%%%%%%%%%

\end{document}